\documentclass[hidelinks,onefignum,onetabnum]{siamart220329} 

\pdfoutput=1

\usepackage{color}

\usepackage{cleveref}
\usepackage{mathtools} 
\usepackage{amsmath}
\newcommand{\dx}{\,{\mathrm d}x}
\newcommand{\ds}{\,{\mathrm d}s}

\newcommand{\SL}{{\mathcal{S}_{\mathcal{L}}}(\tau)}
\newcommand{\SN}{{\mathcal{S}_{\mathcal{N}}}(\tau)}
\newcommand{\SLt}{{\mathcal{S}_{\mathcal{L}}\left(\frac{\tau}{2}\right)}}

\usepackage{lipsum}
\usepackage{amsfonts}
\usepackage{graphicx}
\usepackage{epstopdf}
\usepackage{algorithmic}

\ifpdf
  \DeclareGraphicsExtensions{.eps,.pdf,.png,.jpg}
\else
  \DeclareGraphicsExtensions{.eps}
\fi

\newsiamremark{remark}{Remark}
\newsiamthm{exmp}{Example}
\newsiamremark{hypothesis}{Hypothesis}
\crefname{hypothesis}{Hypothesis}{Hypotheses}
\newsiamthm{claim}{Claim}

\headers{Strang splitting with Neumann boundary condition}{C. Quan, Z. Tan and Y. Wu}

\title{Stability and Convergence of Strang Splitting Method for the Allen--Cahn Equation with Homogeneous Neumann Boundary Condition\thanks{Corresponding authors: Chaoyu Quan and Zhijun Tan}
}

\author{Chaoyu Quan\thanks{School of Science and Engineering,  The Chinese University of Hong Kong,  Shenzhen,  518172,Guangdong, People's Republic of China
    (\email{quanchaoyu@cuhk.edu.cn}).}
        \and Zhijun Tan\thanks{School of Computer Science and Engineering, Sun Yat-sen University, Guangzhou 510006, People's Republic of China
(\email{tzhij@mail.sysu.edu.cn}).}
\and Yanyao Wu\thanks{School of Computer Science and Engineering, Sun Yat-sen University, Guangzhou 510006, People's Republic of China
(\email{wuyy96@mail2.sysu.edu.cn}).}
}

\usepackage{amsopn}

\ifpdf
\hypersetup{
  pdftitle={Stability and convergence analysis of Strang splitting method for the Allen--Cahn equation with homogeneous Neumann boundary condition},
  pdfauthor={C. Quan, Z. Tan, and Y. Wu}
}
\fi

\begin{document}

\maketitle

\begin{abstract}
The Strang splitting method has been widely used to solve nonlinear reaction-diffusion equations, with most theoretical convergence analysis assuming periodic boundary conditions. However, such analysis presents additional challenges for the case of  homogeneous Neumann boundary condition. In this work the Strang splitting method with variable time steps is investigated for solving the Allen--Cahn equation with homogeneous Neumann boundary conditions. Uniform $H^k$-norm stability is established under the assumption that the initial condition $u^0$ belongs to the Sobolev space $H^k(\Omega)$ with integer $k\ge 0$, using the Gagliardo--Nirenberg interpolation inequality and the Sobolev embedding inequality. Furthermore, rigorous convergence analysis is provided in the $H^k$-norm for initial conditions $u^0 \in H^{k+6}(\Omega)$, based on the uniform stability. Several numerical experiments are conducted to verify the theoretical results, demonstrating the effectiveness of the proposed method.
\end{abstract}



\section{Introduction}
In this work, we consider the following Allen--Cahn equation \cite{allen1979microscopic}  with  homogeneous Neumann boundary condition
    \begin{equation}
        \left\{
        \begin{aligned}
            &\partial_t u =\varepsilon^2 \Delta u-f(u), && (t,x)\in (0,T)\times \Omega, \\
            & u(0,x) = u^0(x), &&\mbox{in } \Omega, \\
            & \frac{\partial u}{\partial \mathbf{n}}= 0, &&\mbox{on } \Gamma,
        \end{aligned}
        \right.
        \label{eq:governing1}
    \end{equation}
    whose energy functional is 
    \begin{align}
        E(u) = \int_{\Omega} \left(\frac{\varepsilon^2}{2}|\nabla u|^2 + F(u)\right)  \dx.
    \end{align}
    Here, the spatial domain $\Omega$ is a bounded open set in ${\mathbb{R}^d}~(d \le 3)$ with a $C^1$ boundary, $\mathbf{n}$ is the unit outward normal vector on the boundary $\Gamma$, $f(u)=u^3-u$ is taken as the derivative of the potential $F(u)=\frac{1}{4}(u^2-1)^2$, and the parameter $\varepsilon^2$ is the mobility constant coefficient. It is well-known that the Allen--Cahn equation is the $L^2$ gradient flow of $E(u)$.
    
    Operator splitting methods have been widely applied to solve differential equations. The key concept of operator splitting is to decompose a differential equation into a sequence of simpler problems \cite{mclachlan2002splitting,Blanes_Casas_Murua_2024}. Two well-known splitting methods are the Lie--Trotter scheme \cite{reed1980methods} and Strang splitting \cite{strang_construction_1968}. 
    For the Hamilton--Jacobi system, Glowinski, Leung and Qian develop an operator splitting method for computing effective Hamiltonians based on the Lie scheme in \cite{glowinski2018simple}, which is applicable to both convex and non-convex Hamiltonians.
    For the Hamiltonian system describing the pendulum, one natural splitting form comes from separating the contributions of the kinetic energy and the potential energy. In this case, the Strang splitting scheme reduces to the St{\"o}rmer--Verlet method \cite{hairer2003geometric}. For time-dependent Schr{\"o}dinger equations, it has been proved that Strang splitting achieves second-order time accuracy when applied to pseudo-spectral discretizations of the Schr{\"o}dinger equation \cite{jahnke2000error,bao2023improved}. Moreover, higher-order splittings have been developed for the semiclassical time-dependent Schr{\"o}dinger equation in \cite{blanes2020high}. Operator splitting methods are also effective within quantum physics, including Gross--Pitaevskii equation \cite{bao2005fourth}, the Dirac equation \cite{bao2020super} and the Klein--Gordon equation \cite{bao2022uniform}. For the reaction-diffusion equations, Liu, Wang and Wang in \cite{liu2021structure} propose and analyze a positivity-preserving, energy-stable numerical scheme for a certain type of reaction-diffusion system, followed by a detailed convergence analysis \cite{liu2022convergence}. In \cite{li2017convergence}, Li, Qiao and Zhang prove the global error estimate in discrete $L^2$-norm for the Strang splitting for the epitaxial growth model with slope selection on uniform time meshes. In \cite{lan2023operator}, Lan et al. use operator splitting to solve the mass-conserving convective Allen--Cahn equation, which preserves the discrete maximum principle and conserves the mass.

    In addition to operator splitting method \cite{cheng2015fast,li2017convergence,li_stability_2022-1,li_stability_2022}, we mention that a variety of structure-preserving numerical schemes have been well-developed for solving gradient flows, including the implicit-explicit (IMEX) methods \cite{chen1998applications,tang2016implicit,fu2022unconditionally}, invariant energy quadratization (IEQ) method \cite{yang2016linear}, scalar auxiliary variable (SAV) method \cite{shen2018scalar,akrivis2019energy}, integrating factor Runge--Kutta (IFRK) method \cite{ju2021maximum,li2021stabilized}, exponential time differencing (ETD) method \cite{du2019maximum,li2021unconditionally,fu2022energy,fu2024higher} and so on.

 Suppose time levels $0=t_0<t_1<t_2<\cdots<t_N=T$ with the time steps $\tau_k \coloneqq t_k-t_{k-1}$ for $1\le k \le N$. The Strang splitting is written as
    \begin{align} \label{eq:StrangNonU}
        u^{n+1} = \mathcal{S}_{\mathcal{L}}\left(\frac{\tau_{n+1}}{2}\right)\mathcal{S}_{\mathcal{N}}\left(\tau_{n+1}\right)\mathcal{S}_{\mathcal{L}}\left(\frac{\tau_{n+1}}{2}\right)u^n,
    \end{align}
    where $\tau_{n+1} >0$ is the time step, $\mathcal{S}_{\mathcal{L}}(\frac{\tau_{n+1}}{2}) =\exp\left(\frac{\tau_{n+1}}{2}\varepsilon^2\Delta\right)$ is the linear propagator with homogeneous Neumann boundary condition and $\mathcal{S}_{\mathcal{N}}$ is the solution operator for the nonlinear part 
    \begin{align} \label{eq:SN(v)}
        \partial_t u = -f(u).
    \end{align}
 Here, $\mathcal{S}_{\mathcal{N}}(\tau_{n+1}) v$ can be calculated explicitly by solving an ordinary differential equation if the initial condition $v$ is given, which is called ``exact splitting'' (see other examples in \cite{bader2011fourier,bernier2021exact}). One difficulty lies in the error estimate between the numerical solution and the exact solution. In \cite{jahnke2000error}, Jahnke and Lubich derive the error bounds of using Strang splitting to solve the linear initial problem $\partial_t u=(A+B)u$ with $A$ generating a strongly continuous semigroup and with bounded $B$ with the periodic boundary condition. However, $\mathcal{S}_{\mathcal{N}}$ is a nonlinear solver in \eqref{eq:StrangNonU}. Still with the periodic boundary condition, Blanes et al. establish the convergence analysis of high-order exponential operator splitting method for nonlinear reaction-diffusion equation 
    \cite[Section 4.3, Theorem 1]{blanes2024splittingComplex}, where Lie derivatives and iterated commutators are discussed.
    In \cite{li_stability_2022-1,li_stability_2022}, the energy dissipation law of the Strang splitting method is established in the case of uniform time step with periodic boundary condition. The authors construct a modified energy close to the original within $\mathcal O(\tau)$. However, since the modified energy is related to the time step, it is nontrivial to generalize such energy analysis to the nonuniform time step case. Without energy dissipation law, the stability analysis in \cite{li_stability_2022-1} will encounter a challenge in proving the $H^k$-norm stability of the numerical solution. Therefore, it is necessary to reconstruct the Sobolev norm stability of Strang splitting method with variable time steps.   
    
    In this work, we aim to establish the uniform stability and second-order convergence of the Strang splitting method under certain initial regularity assumptions for the Allen--Cahn equation with homogeneous Neumann boundary condition, rather than the periodic boundary condition. Note that compared to the case of periodic boundary, when using the integration by parts for Neumann boundary condition, some additional boundary integrals could appear with high-order derivatives and consequently shall be estimated further. Our investigation addresses three main aspects. Firstly, we prove the $H^k$-norm regularity of the Allen--Cahn equation using the Gagliardo--Nirenberg interpolation inequality and the Sobolev embedding inequalities. Secondly, we establish a uniform $H^k$-norm bound of $u^n$. This consequently provides an energy bound of the numerical solution.
    Thirdly, we rigorously prove the $H^k$-norm convergence of the Strang splitting method for the Allen--Cahn equation, under the regularity assumptions of initial condition, i.e., $u^0\in H^{k+6}(\Omega)$. In addition, we conduct some experiments to verify the convergence rate and show the efficiency of the adaptive time-stepping strategy. Note that the notations $C_i,~i\in \mathbb{N}$ are general constants, which might vary in different lemmas and theorems.

    This paper is organized as follows. In Section \ref{regularity}, we prove the regularity of the exact solution to Allen--Cahn equation with homogeneous Neumann boundary condition. In Section \ref{stability analysis}, we provide the $L^\infty$-norm and $H^k$-norm results of the numerical solution from the Strang splitting method. Rigorous $H^k$-norm error estimates are presented in Section \ref{section_error_estimate}. Numerical simulations are carried out in Section \ref{simulations}. We draw some conclusions in Section \ref{section_conclusion}.

\section{Regularity of the exact solution} \label{regularity}
We first prove the regularity of the exact solution to \eqref{eq:governing1}. The process and results of this proof are related to the stability and convergence discussed in the subsequent sections. For simplicity, we set $\mathcal{L}\coloneqq \varepsilon^2 \Delta$ in the following content. We use the  definition of $H^k$ norm as 
 \begin{equation}
     \left\| u \right\|_{H^k}^2 \coloneqq \sum_{|\alpha|\le k} \left\|D^{\alpha}u\right\|_{L^2}^2,
 \end{equation} where $\alpha$ is the multi-index of order $|\alpha|$. Let $\alpha = (i_1,i_2,\cdots,i_d)\in \mathbb{N}^d$, then $D^{\alpha} u(t,x)$ can be written as 
\begin{equation}
    \begin{aligned} \label{eq:multiindex}
 D^{\alpha} u(t,x):=\frac{\partial ^{|\alpha|}u(t,x)}{\partial x_1^{i_1} \partial x_2^{i_2}\cdots \partial x_d^{i_d}}.   
    \end{aligned}
\end{equation} 
\begin{theorem}\label{lemma_derivative}
    Assume $u^0 \in H^k(\Omega)$ with $k \ge 0$ and $\|u^0\|_{L^{\infty}}\le 1$. There exists a constant $C\ge 0$  depending on $(\Omega,d,T,\|u^0\|_{H^{k}})$, such that the solution $u(t)$ to \eqref{eq:governing1} satisfies
    \begin{equation} 
        \begin{aligned} \label{eq:Hkresult}
            \left\|u(t)\right\|_{H^k} \le e^{Ct}\left\|u^0\right\|_{H^k},  \quad \forall t \in (0,T].
        \end{aligned}
    \end{equation}
\end{theorem}
\begin{proof}
We prove this by mathematical induction. It is sufficient to prove that there exists a constant $C_1\ge 0$  depending on $(\Omega,d,T,\|u^0\|_{H^{k}})$, such that the solution $u(t)$ to \eqref{eq:governing1} satisfies
    \begin{equation} \label{eq:induction}
        \begin{aligned}
            \frac{1}{2}\partial_t\left\|u(t)\right\|_{H^k}^2 \le C_1\left\|u(t)\right\|_{H^k}^2,  \quad \forall t \in (0,T],
        \end{aligned}
    \end{equation}
which will lead to the desired result \eqref{eq:Hkresult}.

It is well known that the exact solution satisfies the maximum principle, i.e. $\|u(t)\|_{L^{\infty}}\le 1,~\forall t\in \left[0,T\right]$, if $\|u^0\|_{L^{\infty}}\le 1$.
For any fixed $k \ge 0$ and multi-index $\alpha$ with $|\alpha|=k$, acting $D^{\alpha}$ on both sides of \eqref{eq:governing1}, we have
    \begin{equation}
        \left\{
        \begin{aligned} \label{eq:innerproduct_1}
            &\partial_t  (D^{\alpha}u) = \mathcal{L}  (D^{\alpha}u)+D^{\alpha}u-D^{\alpha}u^3, && (t,x)\in (0,T]\times \Omega, \\
            &  D^{\alpha}u(0,x) =  D^{\alpha}u^0(x), &&\mbox{on } \Gamma.
        \end{aligned}
        \right.
    \end{equation}
    Taking inner product with $D^{\alpha
    }u$ in the first equation of \eqref{eq:innerproduct_1}, we have
    \begin{equation} \label{eq:regularuty_1}
        \begin{aligned}
            \frac{1}{2} \partial_t\left\|D^{\alpha}u\right\|_{L^2}^2 =&  -\varepsilon^2\left\|\nabla  (D^{\alpha}u)\right\|_{L^2}^2+\left\| D^{\alpha}u\right\|_{L^2}^2-\left\langle D^{\alpha}u^3,  D^{\alpha}u\right\rangle.
        \end{aligned}
    \end{equation}
    The case of $k=0$ implies \begin{equation} \label{eq:k0}
        \begin{aligned}
            \frac{1}{2} \partial_t\left\|u\right\|_{L^2}^2 \le &  \left\| u\right\|_{L^2}^2.
        \end{aligned}
    \end{equation}
    Solving this differential inequality quickly leads to \eqref{eq:Hkresult}.
    From \eqref{eq:regularuty_1}, we have \begin{equation}\label{eq:regularity}
        \begin{aligned} 
\frac{1}{2} \partial_t\left\|D^{\alpha}u\right\|_{L^2}^2 
 \le&\left\|D^\alpha u\right\|_{L^2}^2-\left\langle 3 u^2 D^\alpha u, D^\alpha u\right\rangle-\sum_{\substack{\alpha_1+\alpha_2+\alpha_3=\alpha \\ \left|\alpha_1\right| ,\left|\alpha_2\right| ,\left|\alpha_3\right| \le k-1}}  \left\langle D^{\alpha_1} uD^{\alpha_2} u D^{\alpha_3} u,D^\alpha u\right\rangle \\
 \le&\left\|D^\alpha u\right\|_{L^2}^2+\sum_{\substack{\alpha_1+\alpha_2+\alpha_3=\alpha \\ \left|\alpha_1\right| ,\left|\alpha_2\right| ,\left|\alpha_3\right| \le k-1}}  \left\|D^{\alpha_1} u D^{\alpha_2} u D^{\alpha_3} u\right\|_{L^2}\left\|D^\alpha u\right\|_{L^2},
        \end{aligned}
    \end{equation}
where $\alpha_1$, $\alpha_2$ and $\alpha_3$ are multi-indices with the same length as $\alpha$. In the following content, we first prove \eqref{eq:induction} for $k=1,2,3$, and then we employ mathematical induction to prove \eqref{eq:induction} for $k\ge 4$. 

We begin with the case of $k=1$ (implying that $D^\alpha = \partial_{x_i}$ for some $1\leq  i\leq d$). Combining \eqref{eq:k0} and \eqref{eq:regularity}, we have 
\begin{equation} 
    \begin{aligned} \label{eq:k1}
        \frac{1}{2} \partial_t\left\|u\right\|_{H^1}^2 \le &  \left\| u\right\|_{H^1}^2. 
    \end{aligned}
\end{equation}
Then we can deduce that 
    \begin{equation}
        \begin{aligned}
           \left\|u(t)\right\|_{H^1} \le e^{t}\left\|u^0\right\|_{H^1}\le e^{T}\left\|u^0\right\|_{H^1},  \quad \forall t \in (0,T].
        \end{aligned}
    \end{equation}

We next consider the case of $k=2$. Using the maximum principle and the interpolation inequality of Gagliardo--Nirenberg \cite[Chapter 9, Comment 3.C, Example 1]{brezis2011functional}:
\begin{align} \label{eq:Gagliardo_Nirenberg}
    \left\|\frac{\partial u}{\partial x_i}\right\|_{L^4}\le C_2\left\|u\right\|_{H^2}^{\frac{1}{2}}\left\|u\right\|_{L^\infty}^{\frac{1}{2}},
\end{align} 
where $C_2$ is a constant depending on $(\Omega,d)$, we obtain 
        \begin{equation} \label{eq:xixj}
            \begin{aligned}
              \left\|\frac{\partial u}{\partial x_i}\frac{\partial u}{\partial x_j} \right\|_{L^2} & =\left(\int_{\Omega} \left(\frac{\partial u}{\partial x_i}\right)^2 \left(\frac{\partial u}{\partial x_j}\right)^2 \dx\right)^{\frac{1}{2}} \\
        &\le \left(\int_{\Omega} \frac{1}{2} \left(\frac{\partial   u}{\partial x_i}\right)^4 + \frac{1}{2} \left(\frac{\partial   u}{\partial x_j}\right)^4 \dx\right)^{\frac{1}{2}} \le C_3\left\|u\right\|_{H^2},
            \end{aligned}
        \end{equation}
where $C_3$ is a constant depending on $(\Omega,d)$.
Using \eqref{eq:xixj} and maximum principle, we have 
 \begin{align} \label{eq:uxiuxju}
    \left\|u\partial_{x_i} u \partial_{x_j} u\right\|_{L^2}&\le\left\|\partial_{x_i} u \partial_{x_j} u  \right\|_{L^2}\le C_3 \left\|u\right\|_{H^2}.
 \end{align}
Then from \eqref{eq:regularity} and \eqref{eq:uxiuxju}, we have
 \begin{equation}
 \begin{aligned} \label{eq:H2D}
     \frac{1}{2} \partial_t\left\|\partial_ {x_{i}x_j} u\right\|_{L^2}^2 &\leq\left\|\partial_{x_i x_j} u\right\|_{L^2}^2+C_4\left\|u\partial_{x_i} u \partial_{x_j}u \right\|_{L^2}\left\|\partial_{x_i x_j} u\right\|_{L^2},\\
& \leq \left\|\partial_{x_i x_j} u\right\|_{L^2}^2+C_3C_4\left\|u\right\|_{H^2}^2,
 \end{aligned}
 \end{equation}
where $C_4\ge 0$ is a constant depending on $d$. 
Summing over with respect to $i$ and $j$ in \eqref{eq:H2D}, we have
 \begin{equation} \label{eq:D}
 \begin{aligned}
     \frac{1}{2} \partial_t\left(\sum_{|\alpha|=2}\left\|D^\alpha u\right\|_{L^2}^2\right) &\leq \sum_{|\alpha|=2}\left\|D^\alpha u\right\|_{L^2}^2+C_5\|u\|_{H^{2}}^2
        \leq C_6\|u\|_{H^{2}}^2,
 \end{aligned}
 \end{equation}
 where $C_5 \ge 0$ and $C_6 \ge 0$ are constants depending on $(\Omega,d)$.
Combining \eqref{eq:k1} and \eqref{eq:D}, we have
\begin{equation} \label{eq:k2}
 \begin{aligned}
     \frac{1}{2} \partial_t\left\| u\right\|_{H^2}^2 \leq C_7\|u\|_{H^{2}}^2,
 \end{aligned}
 \end{equation}
 where $C_7 \ge 0$ is a constant depending on $(\Omega,d)$.
We then conclude that  
\begin{align} \label{eq:uH2}
   \|u(t)\|_{H^2}\le e^{C_7t}\|u^0\|_{H^2} \le e^{C_7T}\|u^0\|_{H^2},\quad \forall t \in (0,T].  
\end{align}

In the case of $k=3$, using the maximum principle, \eqref{eq:xixj}, \eqref{eq:uH2} and the following Sobolev embedding inequality \cite[Section 5.6.3, Theorem 6]{evans_partial_2022} for $d\leq 3$:
$$
    H^2(\Omega) \hookrightarrow L^{\infty}(\Omega),
$$
we have
\begin{equation} \label{eq:2inequality}
    \begin{aligned}
        &\left\|\partial_{x_i} u\partial_{x_j} u \partial_{x_m} u \right\|_{L^2} \le \left\|\partial_{x_i}u\partial_{x_j}u\right\|_{L^2} \left\|\partial_{x_m} u \right\|_{L^\infty} \le C_8  \left\| u \right\|_{H^3}, \\
        &\left\|u\partial_{x_ix_j} u \partial_{x_m} u \right\|_{L^2} \le \left\|\partial_{x_ix_j}u\right\|_{L^2} \left\|\partial_{x_m} u \right\|_{L^\infty} \le C_9  \left\| u \right\|_{H^3}, 
    \end{aligned}
\end{equation}
where $C_8$ and $C_9$ depend on $(\Omega,d,T,\|u^0\|_{H^2})$. From \eqref{eq:regularity} and \eqref{eq:2inequality}, we have
 \begin{equation}
 \begin{aligned}
\frac{1}{2} \partial_t\left\|D^{\alpha} u\right\|_{L^2}^2 & \leq \left\|D^{\alpha} u\right\|_{L^2}^2+C_{10}\left\|D^{\alpha} u\right\|_{L^2}\left\| u \right\|_{H^3},
 \end{aligned}
 \end{equation}
where $C_{10}$ depends on $(\Omega,d,T,\|u^0\|_{H^2})$.
Using similar strategy as the case of $k=2$, we have
\begin{equation} \label{eq:k3}
 \begin{aligned}
     \frac{1}{2} \partial_t\left\| u\right\|_{H^3}^2 \leq C_{11}\|u\|_{H^{3}}^2,
 \end{aligned}
 \end{equation}
where $C_{11}$ depends on $(\Omega,d,T,\|u^0\|_{H^2})$.
We then conclude that  
\begin{align} \label{eq:uH3}
   \|u(t)\|_{H^3} \le e^{C_{11}t}\|u^0\|_{H^3}\le e^{C_{11}T}\|u^0\|_{H^3},\quad \forall t \in (0,T].  
\end{align}

Suppose that for some fixed $k \geq 3$, there exists a constant depending only on $(\Omega,d,T,\|u^0\|_{H^k})$, such that \eqref{eq:induction} holds true. We now prove the case of $|\alpha|=k+1$. 
Without loss of generality, we assume $|\alpha_3|\le|\alpha_2|\le|\alpha_1|\le k$. In the cases of $|\alpha_1|\le k-1$, using the Sobolev embedding inequality 
\cite[Section 5.6.3, Theorem 6]{evans_partial_2022} for $d\leq 3$:
$$
    H^1(\Omega) \hookrightarrow L^{6}(\Omega),
$$
we have
\begin{equation} \label{eq:inequality1}
    \begin{aligned}
         \left\|D^{\alpha_1} u D^{\alpha_2} u D^{\alpha_3} u\right\|_{L^2} &\le \left(\int_{\Omega} \frac{\left(D^{\alpha_1} u\right)^6+\left(D^{\alpha_2} u\right)^6+\left(D^{\alpha_3} u\right)^6}{3}\dx\right)^{\frac{1}{2}}
    \le C_{12}\|u\|_{H^{k}}^3\\
    &\le C_{13}\|u\|_{H^{k+1}},
    \end{aligned}
\end{equation}
where $C_{12}$ is the Sobolev constant and $C_{13}$ depends on $(\Omega,d,T,\|u^0\|_{H^{k}})$. 
In the case of $|\alpha_1|=k$, $|\alpha_2|=1$, $|\alpha_3|=0$, using maximum principle and the Sobolev embedding inequality 
\cite[Section 5.6.3, Theorem 6]{evans_partial_2022} for $d\leq 3$:
$$
    H^2(\Omega) \hookrightarrow L^{\infty}(\Omega),
$$ we have
\begin{equation} \label{eq:inequality2}
    \begin{aligned}
       \left\|D^{\alpha_1} u D^{\alpha_2} u D^{\alpha_3} u\right\|_{L^2} &\le \left\|D^{\alpha_1} u \right\|_{L^2}\left\|D^{\alpha_2} u \right\|_{L^\infty}\left\|D^{\alpha_3} u \right\|_{L^\infty}\le C_{14}\left\|u \right\|_{H^{k}}\left\|u \right\|_{H^{3}}\\
       &\le C_{15}\left\|u \right\|_{H^{k+1}}, 
    \end{aligned}
\end{equation}
where $C_{14}$ is the Sobolev constant and $C_{15}$ depends on $(\Omega,d,T,\|u^0\|_{H^{3}})$.
From \eqref{eq:regularity}, \eqref{eq:inequality1} and \eqref{eq:inequality2}, we have
\begin{equation} \label{eq:D^alpha}
\begin{aligned}
& \frac{1}{2} \partial_t\left\|D^\alpha u\right\|_{L^2}^2 
 \le\left\|D^\alpha u\right\|_{L^2}^2+  C_{16}\left\|u \right\|_{H^{k+1}}\left\|D^\alpha u\right\|_{L^2},
\end{aligned}
\end{equation}
where $C_{16} \ge 0$ depends on $(\Omega,d,T,\|u^0\|_{H^{k}})$. 
Using the induction hypothesis, we have
\begin{equation}
    \begin{aligned}
    \frac{1}{2} \partial_t\left\| u\right\|_{H^{k+1}}^2 \leq C_{17}\|u\|_{H^{k+1}}^2, 
    \end{aligned}
\end{equation}
where $C_{17}\ge 0$ depends on $(\Omega,d,T,\|u^0\|_{H^{k}})$. Therefore, we conclude that
\begin{align} \label{eq:uHk+1}
   \|u(t)\|_{H^{k+1}} \le e^{C_{17}t}\|u^0\|_{H^{k+1}}\le e^{C_{17}T}\|u^0\|_{H^{k+1}},\quad \forall t \in (0,T],  
\end{align}
which completes the proof.
\end{proof}
    
\section{Stability analysis}
    \label{stability analysis}
    In this section, we give two stability results of the numerical solution, namely, the $L^{\infty}$-norm stability and the $H^{k}$-norm stability.
    \subsection{Maximum principle}\label{section_MBP}
    First, we prove the maximum principle of the numerical solution of Strang splitting method with variable time steps.
    \begin{theorem}
        For the numerical solution of Strang splitting \eqref{eq:StrangNonU}, if $\left\|u^0\right\|_{L^{\infty}} \le 1$, it holds that $\sup_{n \ge 1} \left\|u^n\right\|_{L^{\infty}} \le 1$.
    \end{theorem}
    \begin{proof}
    First, it is easy to see that for any $\tau_{n+1}>0$,
    \begin{equation}
        \left\|\tilde u^n \right\|_{L^{\infty}} = \left\|\mathcal{S}_{\mathcal{L}}\left(\frac{\tau_{n+1}} {2}\right) u^n \right\|_{L^{\infty}}  \le \left\| u^n \right\|_{L^{\infty}}
    \end{equation}
as $\mathcal{S}_{\mathcal{L}}(\frac{\tau_{n+1}} 2)$ is a contraction operator (see for example \cite{du_maximum_2021}). 
    We then show that $\mathcal{S}_{\mathcal{N}}(\tau_{n+1})$ also preserves the maximum principle. It is easy to find that \eqref{eq:SN(v)} with initial condition $u|_{t=0} = \tilde{u}^{n}$ and final time $\tau_{n+1}$ can be solved explicitly and the solution is
    \begin{align} \label{eq:SN}
        \mathcal{S}_{\mathcal{N}}(\tau_{n+1})\tilde{u}^{n} = \frac{e^{\tau_{n+1}}\tilde{u}^{n}}{\sqrt{1+(e^{2\tau_{n+1}}-1)(\tilde{u}^n)^2}}.
    \end{align}
    If $\left\| \tilde u^n \right\|_{L^{\infty}} \le 1$, we have
    \begin{align} 
        \left\| \mathcal{S}_{\mathcal{N}}(\tau_{n+1})\tilde{u}^{n} \right\|_{L^{\infty}} \le 1. \label{eq:Sn} 
    \end{align}
    As a consequence, if $\left\| u^n \right\|_{L^{\infty}} \le 1$, we have
    \begin{equation}
    \begin{aligned}
        \left\|u^{n+1}\right\|_{L^{\infty}} =& \left\|\mathcal{S}_{\mathcal{L}}\left(\frac{\tau_{n+1}}{2}\right)\mathcal{S}_{\mathcal{N}}\left(\tau_{n+1}\right)\mathcal{S}_{\mathcal{L}}\left(\frac{\tau_{n+1}}{2}\right)u^n \right\|_{L^{\infty}} \\
        \leq &\left\|\mathcal{S}_{\mathcal{N}}\left(\tau_{n+1}\right)\mathcal{S}_{\mathcal{L}}\left(\frac{\tau_{n+1}}{2}\right)u^n \right\|_{L^{\infty}}  \leq 1. \label{Eq:MBP}
    \end{aligned}
    \end{equation}
    Since $\| u^0 \|_{L^{\infty}}\leq 1$, we then conclude that $\|u^n\|_{L^{\infty}}\leq 1$ holds for all $n\geq 0$.
The maximum principle for \eqref{eq:StrangNonU} is proved. 
\end{proof}

\subsection{$H^k$-norm stability}
\label{H^kstability}
We then provide a uniform $H^k$-norm estimate of the numerical solution $u^n$. For simplicity, we denote by $u(t)\coloneqq u(t,\cdot)$ the exact solution of \eqref{eq:governing1}.
\begin{theorem}
        Assume $u^0 \in H^{k}(\Omega)$ with $k\ge 0$ and $\left\|u^0\right\|_{\infty} \le 1$. For any $1\leq n\leq N$, the numerical solution $u^n$ of the Strang splitting method \eqref{eq:StrangNonU} satisfies
        \begin{align}
             \left\| u^{n} \right\|_{H^k} \le e^{C T} \left\| u^{0} \right\|_{H^k},
        \end{align}
        where $C>0$ is a constant depending only on $(\Omega,d,T,\left\| u^{0} \right\|_{H^k})$.
 \end{theorem}
    
    \begin{proof}
   For any $v\in H^{1}(\Omega)$ with $\left\|v\right\|_{L^\infty}\leq 1$, we have
    \begin{align}
        \left\| \mathcal{S}_{\mathcal{N}}(\tau_{n+1})v \right\| _{L^2} &= \left\|\frac{e^{\tau_{n+1}}v}{\sqrt{1+(e^{2\tau_{n+1}}-1)v^2}}\right\|_{L^2} \le e^{\tau_{n+1}}\left\|  v \right\| _{L^2}, \label{eq:SNnorm1}\\
        \left\| \nabla (\mathcal{S}_{\mathcal{N}}(\tau_{n+1})v) \right\| _{L^2} &= \left(\sum_{i=1}^d\left\|\frac{e^{\tau_{n+1}}}{(1+(e^{2\tau_{n+1}}-1)v^2)^{\frac{3}{2}}}\frac{\partial v}{\partial x_i}\right\|_{L^2}^2 \right)^{\frac{1}{2}}\le e^{\tau_{n+1}}\left\| \nabla v \right\| _{L^2}, \label{eq:SNnorm2}
    \end{align}
    which implies
    \begin{align} 
        \left\| \mathcal{S}_{\mathcal{N}}(\tau_{n+1})v \right\| _{H^1} \le e^{\tau_{n+1}}\left\|  v \right\| _{H^1}. \label{eq:H_1}
    \end{align}
Then we have
    \begin{equation}
\begin{aligned}
        \left\| u^{n+1} \right\|_{H^1}&=\left\| \mathcal{S}_{\mathcal{L}}\left(\frac{\tau_{n+1}}{2}\right)\mathcal{S}_{\mathcal{N}}(\tau_{n+1})\tilde{u}^n\right\|_{H^1} \\
        &\le \left\| \mathcal{S}_{\mathcal{N}}(\tau_{n+1})\tilde{u}^n\right\|_{H^1}\le e^{\tau_{n+1}}\left\|  \tilde{u}^n \right\| _{H^1} \le e^{\tau_{n+1}}\left\|  u^n \right\|_{H^1},
    \end{aligned}
    \end{equation}
    which yields the following
    \begin{align} \label{eq:ODEH1}
        \left\| u^{n} \right\|_{H^1} \le e^{T} \left\| u^{0} \right\|_{H^1}. 
    \end{align} 
Note that $\mathcal{S}_{\mathcal{N}}(\tau)v$ is the exact solution of the following equation at $t=\tau$:
\begin{equation} \label{eq:ODE}
        \left\{
        \begin{aligned}
            &\partial_t u =u-u^3, && t\in (0,\tau), \\
            & u(0) = v, &&\mbox{in } \Omega.
        \end{aligned}
        \right.
    \end{equation}
Using the same strategy in proving the regularity of the exact solution to \eqref{eq:governing1} in Theorem \ref{lemma_derivative}, we can find a constant $C\ge0$ depending only on $(\Omega,d,T,\left\| u^{0} \right\|_{H^k})$, such that
  for any $v\in H^{k}(\Omega)$ with $\|v\|_{L^\infty}\leq 1$,
  \begin{equation}
   \begin{aligned} \label{eq:regularity_ODE}
        \left\| \mathcal{S}_{\mathcal{N}}(\tau)v \right\| _{H^k} \le e^{C\tau}\left\|  v \right\| _{H^k}. 
    \end{aligned}
  \end{equation}
Since $\SL:H^k(\Omega)\longrightarrow H^k(\Omega)$ is a contraction operator, i.e. $\left\|\SL\right\|_{H^k}\le 1,~\forall \tau\ge 0$, we have  
   \begin{equation}
\begin{aligned}
        \left\| u^{n+1} \right\|_{H^k}&=\left\| \mathcal{S}_{\mathcal{L}}\left(\frac{\tau_{n+1}}{2}\right)\mathcal{S}_{\mathcal{N}}(\tau_{n+1})\tilde{u}^n\right\|_{H^k} \le \left\| \mathcal{S}_{\mathcal{N}}(\tau_{n+1})\tilde{u}^n\right\|_{H^k}\\
        &\le e^{C\tau_{n+1}}\left\|  \tilde{u}^n \right\| _{H^k} \le e^{C\tau_{n+1}}\left\|  u^n \right\|_{H^k}\le e^{CT}\left\|  u^0 \right\|_{H^k}.
    \end{aligned}
    \end{equation}
    \end{proof}
\begin{remark}
 Based on the $H^1$-norm stability in the previous subsection, we estimate the original energy for numerical solutions. Assume $u^0 \in H^1(\Omega)$. As a result of \eqref{eq:ODEH1}, we obtain an upper bound of $E(u^n)$:
    \begin{equation}
        \begin{aligned}
            E\left(u^n\right) &= \int_{\Omega} \left(\frac{\varepsilon^2}{2}|\nabla u^n|^2 + \frac{1}{4}\left(1-\left(u^n\right)^2\right)^2\right)  \dx \\
            & \le  \frac{\varepsilon^2}{2} \left\|u^n\right\|_{H^1}^2+\frac{1}{4}\left|\Omega\right| \le \frac{\varepsilon^2}{2} e^{2T}\left\|u^0\right\|_{H^1}^2+\frac{1}{4}\left|\Omega\right|.
        \end{aligned}
    \end{equation}  
In \cite{li_stability_2022-1}, a modified energy is constructed for the Strang splitting with uniform time steps. However, this modified energy is related to the time step, it is still unknown to establish the energy dissipation law for the nonuniform time step case.
\end{remark}

\begin{remark}
     Recently, there are some works on the stability of the variable step numerical schemes. For example, in \cite{chen2019second}, Chen et al. present and analyze the BDF2 numerical scheme for the Cahn--Hilliard equation on the nonuniform time meshes, incorporating a novel generalized discrete Gr{\"o}nwall-type inequality. Later, some stability results of the adaptive BDF2 scheme are established for the Allen--Cahn equation \cite{liao_energy_2020} and the diffusion equations \cite{liao2021analysis}.
\end{remark}

\section{Error estimates} \label{section_error_estimate}
    In this section, rigorous $H^k$-norm convergence is given for the Strang splitting for \eqref{eq:governing1} with variable time steps and homogeneous Neumann boundary condition. We denote by $\mathcal{T}(\tau)$ the solution operator mapping $u(0)$ to $u(\tau)$ for \eqref{eq:governing1}. We define the Strang splitting operator $\mathcal{S}(\tau_{n+1})$ as follows:
\begin{equation} \label{eq:S(tau)}
    \begin{aligned}
       \mathcal{S}(\tau_{n+1})\coloneqq \mathcal{S}_{\mathcal{L}}\left(\frac{\tau_{n+1}}{2}\right)\mathcal{S}_{\mathcal{N}}\left(\tau_{n+1}\right)\mathcal{S}_{\mathcal{L}}\left(\frac{\tau_{n+1}}{2}\right).
    \end{aligned}
\end{equation}
The maximum time step is defined as
\begin{equation} \label{eq:tau_max}
    \begin{aligned}
        \tau_{\max} \coloneqq \underset{1\le k \le N}{\max} \tau_k.
    \end{aligned}
\end{equation}
To estimate the $H^k$-norm error, we introduce some lemmas in the following. 
To simplify the notations, we denote 
\begin{equation} \label{eq:sl_Q}
    \begin{aligned}
        \mathcal{Q}(v)&\coloneqq \mathcal{S}_{\mathcal{L}}\left(\frac{\tau}{2}\right)v-\left(\mathcal{I}+\frac{\tau}{2} \mathcal{L}\right)v.
    \end{aligned}
\end{equation}
Note that the operator $\mathcal{Q}=(\tau\mathcal{L})^2\varphi_2(\tau \mathcal{L})$ with $\varphi_2(z)=\frac{e^z-1-z}{z^2}$.
Here, we define
\begin{equation} \label{eq:exponential operator}
    \begin{aligned}
      \varphi_k(z)\coloneqq\int_{0}^1e^{(1-\theta)z}\frac{\theta^{k-1}}{(k-1)!}\mathrm{d}\theta  
    \end{aligned}
\end{equation}
for $k\ge 1$ and $\varphi_0(z)=e^z$ as in \cite{hochbruck2010exponential}, which are commonly used in exponential integrators.
Next, let us estimate the Sobolev norms of $\mathcal{Q}(v)$.
\begin{lemma}{(Property of $\mathcal{Q}$)} \label{propertyQ}
    For any $v\in H^{4+|\alpha|}(\Omega)$, $ \left\| D^{\alpha} (\mathcal{Q}(v)) \right\|_{L^2}\le \frac{\tau^2 }{8} \left\|  D^{\alpha} (\mathcal{L}^2 v)\right\|_{L^2}$.
\end{lemma}
\begin{proof}
    First, consider the case of $|\alpha|=0$. We define
    \begin{equation}
        \begin{aligned}
            \phi(z) \coloneqq e^z-1-z.
        \end{aligned}
    \end{equation}
    We note that $\phi(z)=z^2\varphi_2(z)$ as in \eqref{eq:exponential operator}.
    By Taylor expansions, we have 
    \begin{equation}
        \begin{aligned}
            \phi(z) & =\frac{z^2}{2} e^{\theta}\le \frac{z^2}{2}, \quad \forall z<0,
        \end{aligned}
    \end{equation}
    where $\theta\in(z,0)$. Let $(\mu_j,\varphi_j)_{j=1}^{\infty}$ be the eigenpairs of the selfadjoint and negative-definite operator $\mathcal{L}$. Here, $\varphi_j$ forms a complete orthogonal basis of $L^2(\Omega)$. Then we have
    \begin{equation}\label{eq:property1}
        \begin{aligned}
            \|\mathcal{Q}(v)\|_{L^2}^2 & =\left\|\phi\left(\frac{\tau}{2} \mathcal{L}\right) v\right\|_{L^2}^2=\sum_{j=1}^{\infty} \phi^2\left(\frac{\tau}{2} \mu_j\right) \left|\left(v, \varphi_j\right)\right|^2 \leq \sum_{j=1}^{\infty} \frac{\tau^4}{64} \mu_j^4 \left|\left(v, \varphi_j\right)\right|^2 \\
& =\frac{\tau^4}{64}\left\|\mathcal{L}^2 v\right\|_{L^2}^2.
        \end{aligned}
    \end{equation}
Next, we consider the case of $|\alpha|\ge 1$. Since $\mathcal{Q}$ and $D^\alpha$ commute, we can easily derive the following from \eqref{eq:property1}:
    \begin{equation} 
    \begin{aligned}
      \left\| D^{\alpha} (\mathcal{Q}(v)) \right\|_{L^2} =  \left\| \mathcal{Q}(D^{\alpha} v) \right\|_{L^2} \le \frac{\tau^2 }{8} \left\|  D^{\alpha} (\mathcal{L}^2 v)\right\|_{L^2}. 
    \end{aligned}
\end{equation}
\end{proof}

\begin{lemma} \label{estimate_M}
Let $g(s)=\int_0^s h(\xi)(s-\xi)^l \mathrm{d} \xi$, where $\|h(\xi)\|_{H^k} \leq C$, for all $\xi \in(0, \tau]$, and $l\ge 1$. Then there exists a constant $C_1$ depending on $(C,m,l)$, such that
\begin{equation}
    \begin{aligned}
        \left\|\int_0^\tau s^m g(s) d s\right\|_{H^k} \leq C_1 \tau^{m+l+2}.
    \end{aligned}
\end{equation}
\end{lemma}
\begin{proof}
Direct calculation gives
\begin{equation}
\begin{aligned}
\|g(s)\|_{H^k}& =\left\|\int_0^s h(\xi)(s-\xi)^l\mathrm{d} \xi\right\|_{H^k} \leq \int_0^s \left\|h(\xi)\right\|_{H^k}(s-\xi)^l\mathrm{d} \xi \leq C\frac{s^{l+1}}{l+1} .
\end{aligned}
\end{equation}
Then we have
\begin{equation}
\begin{aligned}
\left\|\int_0 ^{\tau} s^m g(s) \ds\right\|_{H^k} &\le\int_0^\tau s^m \left\|g(s)\right\|_{H^k} \ds  \leq C\frac{\tau^{m+l+2}}{(l+1)(m+l+2)} .
\end{aligned}
\end{equation}
\end{proof}
\begin{lemma}  \label{convergence_1}
    Assume $v \in H^{k+6}(\Omega)$, $\|v\|_{L^{\infty}}\le 1$, and $0<\tau \le 1$. There exists a positive constant $C$, depending only on $(\Omega,d,T,\|v\|_{H^{k+6}},\varepsilon)$, such that
    \begin{equation}
    \begin{aligned}
        & \left\| \mathcal{S}(\tau)v -  \SL v+\tau f(v)+\frac{\tau^2}{2} \mathcal{L} f(v)+\frac{\tau^2}{2} f^{\prime}(v) \mathcal{L} v-\frac{\tau^2}{2} f(v) f^{\prime}(v) \right\|_{H^k} \le C \tau^3, 
    \end{aligned}
\end{equation}
where $\mathcal{S}(\tau)$ is the Strang splitting operator defined in \eqref{eq:S(tau)}.
\end{lemma}

\begin{proof}
From \eqref{eq:S(tau)}, we know
$\mathcal{S}(\tau) v=\mathcal{S}_{\mathcal{L}}\left(\frac{\tau}{2}\right)\mathcal{S}_{\mathcal{N}}\left(\tau\right)\mathcal{S}_{\mathcal{L}}\left(\frac{\tau}{2}\right)v$. 
Let us define $$w(0)\coloneqq \mathcal{S}_{\mathcal{L}}\left(\frac{\tau}{2}\right)v.$$ 
Here, we use $w(t)$ to denote $w(t,\cdot)$ as before.
The calculation of $\mathcal{S}(\tau) v$ consists of the following three steps. Firstly, we calculate the Taylor expansion of $\mathcal{S}_{\mathcal{N}}(\tau)w(0)$ in the time direction at $w(0)$. Next, we substitute $w(0)$ with $\mathcal{S}_{\mathcal{L}}\left(\frac{\tau}{2}\right)v$. Finally, we apply $\SLt$ on $\SN w(0)$. This completes the process of calculating $\mathcal{S}(\tau)v$.

\text{Step 1.} Recall that $w(\tau)=\SN w(0)$ satisfies the following equation
\begin{equation}
 \left\{
        \begin{aligned}
       &\partial_t w = -f(w), \quad 0 \le t \le \tau,\\
       &w|_{t=0} = w(0),
        \end{aligned}
        \right.
        \label{secondorder1}
\end{equation}
where $\left\|w(0)\right\|_{L^\infty}\le 1$. It is not difficult to check that
\begin{equation}
    \begin{aligned}
       \partial_{tt} w &= f'(w)f(w),\\
       \partial_{ttt} w &= -f''(w)f^2(w)-(f'(w))^2f(w). 
    \end{aligned}
\end{equation}
By Taylor expansions, we have 
\begin{equation} \label{eq:w(tau)}
    \begin{aligned}
        w(\tau) =& w(0) - \tau f(w(0))+\frac{\tau^2}{2} f'(w(0))f(w(0))+\mathcal{R}_1,
    \end{aligned}
\end{equation}
where 
\begin{equation}
    \begin{aligned}
        \mathcal{R}_1&=\frac{1}{2}\int_0^{\tau}\partial_{ttt}w(s)(\tau-s)^2 \ds \\
        & = \frac{1}{2}\int_0^{\tau} \left(-6w(s)(w^3(s)-w(s))^2-(3w^2(s)-1)^2(w^3(s)-w(s))\right)(\tau-s)^2\ds.
    \end{aligned}
\end{equation}

\text{Step 2.} Replacing $w(0)$ with $\SLt v$ in \eqref{eq:w(tau)}, we have
\begin{equation}
    \begin{aligned}
        \SN \SLt v =& \SLt v - \tau f\left(\SLt v\right)+\frac{\tau^2}{2} f'\left(\SLt v\right)f\left(\SLt v\right)+\mathcal{R}_1.
    \end{aligned}
\end{equation}

\text{Step 3.} Acting $\mathcal{S}_{\mathcal{L}}(\frac{\tau}{2})$ on both sides of the equation above, we obtain
\begin{equation} \label{eq:u}
    \begin{aligned}
        \mathcal{S}(\tau)v =&\SLt w(\tau)\\
        =&\SL v - \tau\SLt f\left(\SLt  v\right) + \frac{\tau^2}{2}\SLt \left(f'\left(\SLt v\right)f\left(\SLt v\right)\right)\\
        &+ \SLt \mathcal{R}_1.
    \end{aligned}
\end{equation}
According to the regularity result of   \eqref{eq:ODE}, Lemma \ref{estimate_M} and the Sobolev embedding theorem
\cite[Section 5.6.3, Theorem 6]{evans_partial_2022} for $d\leq 3$:
\begin{equation} \label{embedding_thm}
    \begin{aligned}
    H^2(\Omega) \hookrightarrow  L^{\infty}(\Omega),
    \end{aligned}
\end{equation} 
we have
\begin{equation}
    \begin{aligned} \label{eq:eta}
       &\left\|\SLt \mathcal{R}_1\right\|_{H^k}\\
       \le & \left\|\mathcal{R}_1\right\|_{H^k}\\ \le& \frac{\tau^3}{6}\|-6w(s)(w^3(s)-w(s))^2-(3w^2(s)-1)^2(w^3(s)-w(s))\|_{H^k}\\
      \le & \frac{\tau^3}{6}\left(\sum_{|\alpha| \le k}\|D^{\alpha}\left(-6w(s)(w^3(s)-w(s))^2-(3w^2(s)-1)^2(w^3(s)-w(s))\right)\|_{L^2}^2\right)^{\frac{1}{2}}\\
       \le& \frac{\tau^3}{6}\left(\sum_{|\alpha| \le k}|\Omega|^2\|D^{\alpha}\left(-6w(s)(w^3(s)-w(s))^2-(3w^2(s)-1)^2(w^3(s)-w(s))\right)\|_{L^\infty}^2\right)^{\frac{1}{2}}\\
      \le&  C_1\tau^3,
    \end{aligned}
\end{equation}
where $C_1$ depends on $(\Omega, d,T,\|v\|_{H^{k+2}})$.
 
Next, we proceed with a more detailed calculation on \eqref{eq:u} via two steps. We first calculate $f\left(\SLt v\right)$ and $f'\left(\SLt v\right)$ as follows
\begin{equation}
\begin{aligned} \label{eq:ut_n3}
        f\left(\SLt v\right)& =f\left(v+\frac{\tau}{2} \mathcal{L} v+\mathcal{Q}(v)\right) =f(v)+\frac{\tau}{2} f^{\prime}(v) \mathcal{L} v+\mathcal{R}_2 =f(v)+\mathcal{R}_3,\\
        f^{\prime}\left(\SLt v\right) & =f^{\prime}\left(v+\frac{\tau}{2} \mathcal{L} v+\mathcal{Q}(v)\right)  =f^{\prime}(v)+\mathcal{R}_4,
    \end{aligned}
\end{equation}    
where 
\begin{align*}
  \mathcal{R}_2=&  f^{\prime}(v) \mathcal{Q}(v)  +\int_v^{\SLt (v)}  \frac{\partial^2 f}{\partial u^2}\left(\SLt (v)-u\right) \mathrm{d} u
  \\=&f^{\prime}(v) \mathcal{Q}(v) +\int_v^{\SLt (v)} 6 u\left(\SLt (v)-u\right) \mathrm{d}u \nonumber\\
  =&f^{\prime}(v) \mathcal{Q}(v)+\frac{3}{4} \tau^2 v(\mathcal{L} v)^2+\frac{\tau^3}{8}(\mathcal{L} v)^3+3 \tau v \mathcal{L} v \mathcal{Q}(v)+3 v(\mathcal{Q}(v))^2\\
  &+\frac{3}{4} \tau^2(\mathcal{L} v)^2 \mathcal{Q}(v)+\frac{3}{2} \tau \mathcal{L} v(\mathcal{Q}(v))^2 +(\mathcal{Q}(v))^3\\
  \mathcal{R}_3=&  \frac{\tau}{2} f^{\prime}(v) \mathcal{L} v+\mathcal{R}_2, \\
  \mathcal{R}_4=& \int_v^{\SLt (v)} \frac{\partial^2 f}{\partial u^2} \mathrm{d} u=3\left(\SLt (v)\right)^2-3v^2\\
  =&3\left(\tau v \mathcal{L} v+2 v \mathcal{Q} (v)+\frac{\tau^2}{4}(\mathcal{L} v)^2+\tau \mathcal{L} v \mathcal{Q} (v)+(\mathcal{Q} (v))^2\right). 
\end{align*}
In the case of $k=0$, 
according to Lemma \ref{propertyQ}, \eqref{embedding_thm} and the definition of $\mathcal R_2$, we have 
\begin{equation} \label{R1_infty}
    \begin{aligned}
\left\|\mathcal{R}_2\right\|_{L^{\infty}} \le & C_{2} \left(\|f^{\prime}(v) \|_{\infty}\|\mathcal{Q}(v)\|_{H^2}+\frac{3}{4} \tau^2\|\mathcal{L} v\|_{H^2}^2+\frac{\tau^3}{8}\|\mathcal{L} v\|_{H^2}^3+3 \tau\|\mathcal{L} v\|_{H^2}\|\mathcal{Q} (v)\|_{H^2}\right.\\
&\left. +3\|\mathcal{Q} (v)\|_{H^2}^2 +\frac{3}{4} \tau^2\|\mathcal{L} v\|_{H^2}^2\|\mathcal{Q} (v)\|_{H^2}+\frac{3}{2} \tau\|\mathcal{L} v\|_{H^2}\|\mathcal{Q} (v)\|_{H^2}^2+\|\mathcal{Q} (v)\|_{H^2}^3\right)\\
\le & C_3\tau^2,
    \end{aligned}
\end{equation}
where $C_2$ depends on $(\Omega,d)$ and $C_3$ depends on $(\Omega, d, T,\|v\|_{H^6},\varepsilon)$.
Similarly, we have 
\begin{equation}
\begin{aligned}\label{eq:R2R3_infty}
    \|\mathcal{R}_3\|_{L^\infty}\leq C_4 \tau \quad \mbox{and}\quad \|\mathcal{R}_4\|_{L^\infty}\leq C_5 \tau,
\end{aligned}
\end{equation}
with $C_4$ and $C_5$ depending on $(\Omega, d, T,\|v\|_{H^6},\varepsilon)$. 

In the case of $k=1$, using \eqref{embedding_thm} and following the same strategy as the case of $k=0$, we have
\begin{equation} \label{R2_H1}
    \begin{aligned}
        \|\mathcal{R}_2\|_{H^1} \le C_6\tau^2,\quad \|\mathcal{R}_3\|_{H^1}\leq C_7 \tau \quad \mbox{and}\quad \|\mathcal{R}_4\|_{H^1}\leq C_8 \tau,
    \end{aligned}
\end{equation}
where $C_6$, $C_7$ and $C_8$ depend on $(\Omega, d, T,\|v\|_{H^7},\varepsilon)$.

In the case of $k\ge2$, using the multiplication theorems in Sobolev spaces \cite[Section 7, Theorem 7.4]{behzadan2021multiplication}:
\begin{equation}
    \begin{aligned} \label{multiplication}
        \|hg\|_{H^k}\le C_{9} \|h\|_{H^k}\|g\|_{H^k},
    \end{aligned}
\end{equation}
where $C_9$ depends only on $(\Omega,d)$,
we have
\begin{equation} \label{R2_Hk}
    \begin{aligned}
      \left\|\mathcal{R}_2\right\|_{H^k} \le& C_{10}\left(\|f^{\prime}(v) \|_{H^k}\|\mathcal{Q}(v)\|_{H^k}+\frac{3}{4} \tau^2\|v\|_{H^k}\|\mathcal{L} v\|_{H^k}^2+3 \tau\|v\|_{H^k}\|\mathcal{L} v\|_{H^k}\|\mathcal{Q} (v)\|_{H^k}\right.\\
& \left.+\frac{\tau^3}{8}\|\mathcal{L} v\|_{H^k}^3+3\|v\|_{H^k}\|\mathcal{Q} (v)\|_{H^k}^2 +\frac{3}{4} \tau^2\|\mathcal{L} v\|_{H^k}^2\|\mathcal{Q} (v)\|_{H^k}+\frac{3}{2} \tau\|\mathcal{L} v\|_{H^k}\|\mathcal{Q} (v)\|_{H^k}^2\right.\\
&\Bigl.+\|\mathcal{Q} (v)\|_{H^k}^3\Bigr)
\le C_{11} \tau^2,
    \end{aligned}
\end{equation}
where $C_{10}$ depends on $(\Omega,d)$ and $C_{11}$ depends on $(\Omega, d, T,\|v\|_{H^{k+6}},\varepsilon)$.
Similarly, we have 
\begin{equation}
    \|\mathcal{R}_3\|_{H^k}\leq C_{12} \tau \quad \mbox{and}\quad \|\mathcal{R}_4\|_{H^k}\leq C_{13} \tau
\end{equation}
with $C_{12}$ and $C_{13}$ depending on $(\Omega, d, T,\|v\|_{H^{k+6}},\varepsilon)$.

Then, from \eqref{eq:sl_Q} and \eqref{eq:ut_n3}, we have
\begin{equation}\label{eq:ut_n4}
    \begin{aligned}
 \SLt f\left(\SLt v\right)=&\left(\mathcal{I}+\frac{\tau}{2} \mathcal{L}+\mathcal{Q}\right)\left(f(v)+\frac{\tau}{2} f^{\prime}(v) \mathcal{L} v+\mathcal{R}_2\right) \\
=&\left(\mathcal{I}+\frac{\tau}{2} \mathcal{L}\right) f(v)+\frac{\tau}{2} f^{\prime}(v) \mathcal{L} v+\mathcal{R}_5,\\  
\SLt\left(f\left(\SLt v\right) f^{\prime}\left(\SLt v\right)\right)=&\left(\mathcal{I}+\frac{\tau}{2} \mathcal{L}+\mathcal{Q}\right)\left(f(v) f^{\prime}(v)+f^{\prime}(v)\mathcal{R}_3 \right.\\
&\left.+f(v)\mathcal{R}_4+\mathcal{R}_3\mathcal{R}_4\right) \\
=&f(v) f^{\prime}(v)+\mathcal{R}_6,
    \end{aligned}
\end{equation}
where
\begin{align}
     \mathcal{R}_5= &\left(\mathcal{I}+\frac{\tau}{2} \mathcal{L}+\mathcal{Q}\right)\mathcal{R}_2+\left(\frac{\tau}{2} \mathcal{L}+\mathcal{Q}\right) \left(\frac{\tau}{2}f^{\prime}(v)\mathcal{L}v\right)+\mathcal{Q}\left(f(v)\right), \nonumber\\
   \mathcal{R}_6=& \left(\mathcal{I}+\frac{\tau}{2} \mathcal{L}+\mathcal{Q}\right)\left(f^{\prime}(v)\mathcal{R}_3+f(v)\mathcal{R}_4+\mathcal{R}_3\mathcal{R}_4\right) +\left(\frac{\tau}{2} \mathcal{L}+\mathcal{Q}\right)\left(f(v) f^{\prime}(v)\right).\nonumber
\end{align}
Similarly as above, according to Lemma \ref{propertyQ}, \eqref{R1_infty}, \eqref{eq:R2R3_infty} and the fact $\SLt =\mathcal{I}+\frac{\tau}{2} \mathcal{L}+\mathcal{Q}$, we have 
\begin{equation} \label{eq:R5R6}
    \begin{aligned}
        \|\mathcal R_5\|_{L^2}
        \leq & \left\|\SLt\right\|_{L^2}\|\mathcal{R}_2\|_{L^2}+\left\|\left(\frac{\tau}{2} \mathcal{L}+\mathcal{Q}\right) \left(\frac{\tau}{2}f^{\prime}(v)\mathcal{L}v\right)\right\|_{L^2} +\left\|\mathcal{Q}\left(f(v)\right)\right\|_{L^2}\\
        \le & \left\|\SLt\right\|_{L^2}\|\mathcal{R}_2\|_{L^2}+\left\|\frac{\tau^2}{4} \mathcal{L} \left(f^{\prime}(v)\mathcal{L}v\right)\right\|_{L^2}+\left\| \frac{\tau^3}{16}\mathcal{L}^2\left(f^{\prime}(v)\mathcal{L}v\right)\right\|_{L^2} \\
        &+\frac{\tau^2}{8}\left\|\mathcal{L}^2f(v)\right\|_{L^2}\le C_{14} \tau^2,\\
        \|\mathcal R_6\|_{L^2}
        \leq & \left\|\SLt\right\|_{L^2}\|f^{\prime}(v)\mathcal{R}_3 +f(v)\mathcal{R}_4+\mathcal{R}_3\mathcal{R}_4\|_{L^2}\\
        &+\left\|\frac{\tau}{2} \mathcal{L}\left(f(v) f^{\prime}(v)\right)\right\|_{L^2}+
        \left\|\frac{\tau^3}{16} \mathcal{L}^2\left(f(v) f^{\prime}(v)\right)\right\|_{L^2}
        \le C_{15} \tau, \\
    \end{aligned}
\end{equation}
where $C_{14}$ and $C_{15}$ depend on $(\Omega, d, T,\|v\|_{H^6},\varepsilon)$. Thus, we obtain
\begin{equation}\label{eq:tau1}
    \begin{aligned}
        \left\| -\tau \mathcal{R}_5 +\frac{\tau^2}{2} \mathcal{R}_6\right\|_{L^2}
        \le  C_{16} \tau^3,
    \end{aligned}
\end{equation}
where $C_{16}$ depends on $(\Omega, d, T,\|v\|_{H^6},\varepsilon)$. Using similar analysis to \eqref{R2_H1} and \eqref{R2_Hk}, we have \begin{equation}\label{eq:tau2}
    \begin{aligned}
        \left\| -\tau \mathcal{R}_5 +\frac{\tau^2}{2} \mathcal{R}_6\right\|_{H^k}
        \le  C_{17} \tau^3,
    \end{aligned}
\end{equation}
where $C_{17}$ depends on $(\Omega, d, T,\|v\|_{H^{k+6}},\varepsilon)$.

Combining \eqref{eq:u}, \eqref{eq:ut_n3} and \eqref{eq:ut_n4}, one can obtain
\begin{equation} \label{eq:u_zhankai}
    \begin{aligned}
        \mathcal{S}(\tau)v=&\SLt w(\tau)v\\  =&\SL v-\tau f(v)-\frac{\tau^2}{2} \mathcal{L} f(v)-\frac{\tau^2}{2} f^{\prime}(v) \mathcal{L} v+\frac{\tau^2}{2} f(v) f^{\prime}(v)\\&-\tau \mathcal{R}_5+\frac{\tau^2}{2} \mathcal{R}_6+\SLt \mathcal{R}_1(\tau).
\end{aligned}
\end{equation}
Combining \eqref{eq:eta}, \eqref{eq:tau1}, \eqref{eq:tau2} and \eqref{eq:u_zhankai}, we have
\begin{equation}
    \begin{aligned}
        & \left\| \mathcal{S}(\tau)v -  \SL v+\tau f(v)+\frac{\tau^2}{2} \mathcal{L} f(v)+\frac{\tau^2}{2} f^{\prime}(v) \mathcal{L} v-\frac{\tau^2}{2} f(v) f^{\prime}(v) \right\|_{H^k} \le C_{18} \tau^3, 
    \end{aligned}
\end{equation}
where $C_{18}$ depends on $(\Omega,d,T,\|v\|_{H^{k+6}},\varepsilon)$. 
\end{proof}

\begin{lemma} \label{convergence_2}
    Assume $u(t)$ is the exact solution of the following PDE:
    \begin{equation} \label{T(tau)}
    \left\{
           \begin{aligned}
          &\partial_t u = \mathcal{L} u -f(u), \quad 0 < t \le \tau,\\
          &u(0,x) = v,\\
           \end{aligned}
           \right.
\end{equation}
where $v \in H^{k+6}(\Omega)$, $\|v\|_{L^{\infty}}\le 1$, and $0<\tau \le 1$. Then there exists a positive constant $C$, depending only on $(\Omega,d,T,\|v\|_{H^{k+6}},\varepsilon)$, such that
    \begin{equation}
    \begin{aligned}
        & \left\| u(\tau) -   \SL v+\tau f(v)+\frac{\tau^2}{2} \mathcal{L} f(v)+\frac{\tau^2}{2} f^{\prime}(v) \mathcal{L} v-\frac{\tau^2}{2} f(v) f^{\prime}(v) \right\|_{H^k} \le C \tau^3.
    \end{aligned}
\end{equation}
\end{lemma}
\begin{proof}
By the Duhamel's principle, we have
\begin{equation}
    \begin{aligned} \label{eq:duhamel}
    \begin{aligned}
u(\tau)= & \mathcal{S}_\mathcal{L}(\tau) v-\int_0^\tau \mathcal{S}_\mathcal{L}(\tau-s) f(u(s)) \ds \\
= & \mathcal{S}_\mathcal{L}(\tau)v-\int_0^\tau f(u(s)) \ds-\int_0^\tau(\tau-s) \mathcal{L} f(u(s)) \ds -\int_0^\tau \mathcal{R}_1(f(u(s))) \ds,
\end{aligned}
    \end{aligned}
\end{equation}
where 
\begin{equation}
    \begin{aligned}
        \mathcal{R}_1(f(u(s))&\coloneqq \left(\mathcal{S}_{\mathcal{L}}\left(\tau-s\right)-\left(\mathcal{I}+(\tau-s) \mathcal{L}\right)\right)f(u(s)).
    \end{aligned}
\end{equation}
Using similar strategy in Lemma \ref{propertyQ}, we have
\begin{equation} \label{eq:R}
    \begin{aligned}
        \|D^{\alpha}\left(\mathcal{R}_1(f(u(s))\right)\|_{L^2}&\le \frac{(\tau-s)^2}{2}\|D^{\alpha}\left(\mathcal{L}^2\left(f\left(u\left(s\right)\right)\right)\right)\|_{L^2},\quad \forall 0 \le s \le \tau.
    \end{aligned}
\end{equation}
Then by the Taylor expansion, we have
\begin{equation} \label{eq:f_taylor}
    \begin{aligned}
       f(u(s)) & =f(u(0))+\frac{\partial f(u)}{\partial s} \Big|_{s=0} s+\mathcal{R}_2(s) \\
& =f(v)+\frac{\partial f}{\partial u}\Big|_{u=v} \frac{\partial u}{\partial s}\Big|_{s=0} s+\mathcal{R}_2(s) \\
& =f(v)+s f^{\prime}(v)(\mathcal{L} v-f(v))+\mathcal{R}_2(s) ,
    \end{aligned}
\end{equation}
where 
\begin{equation} \label{eq:f_tt}
    \begin{aligned}
       \mathcal{R}_2(s)&= \int_0^s \frac{\partial^2 f(u(\xi))}{\partial \xi^2}(s-\xi) \mathrm{d} \xi\\
       &= \int_0^s\left(- \partial_{\xi\xi}u(\xi)+6 u(\xi)\left(\mathcal{L}(u(\xi))-f(u(\xi))\right)^2+3 u^2(\xi) \partial_{\xi\xi}u(\xi)\right)(s-\xi) \mathrm{d} \xi,
    \end{aligned}
\end{equation}
with 
\begin{align*}
        \partial_{\xi\xi}u(\xi)=\mathcal{L}^2u(\xi)-\mathcal{L}f(u(\xi))-f^{\prime}(u(\xi))\mathcal{L}u(\xi)+f^{\prime}(u(\xi))f(u(\xi)).
\end{align*}
Substituting \eqref{eq:f_taylor} into \eqref{eq:duhamel}, we have 
\begin{equation}
    \begin{aligned} \label{eq:int_f}
        \int_0^\tau f(u(s)) \ds=\tau f(v)+\frac{\tau^2}{2} f^{\prime}(v)(\mathcal{L} v-f(v))+\int_0^\tau \mathcal{R}_2(s) \ds,
    \end{aligned}
\end{equation}
and
\begin{equation}
    \begin{aligned} \label{eq:int_lf}
        \int_0^\tau(\tau-s) \mathcal{L} f(u(s)) \ds & =\int_0^\tau(\tau-s) \mathcal{L}\left(f(v)+s f^{\prime}(v)(\mathcal{L} v-f(v))+\mathcal{R}_2(s)\right) \ds \\
& =\frac{\tau^2}{2} \mathcal{L} f(v)+\frac{\tau^3}{6} \mathcal{L}\left(f^{\prime}(v)(\mathcal{L} v-f(v)\right)+\int_0^\tau(\tau-s) \mathcal{L} \mathcal{R}_2(s) \ds.
    \end{aligned}
\end{equation}
To estimate $\left\|\int_0^\tau \mathcal{R}_2(s) \ds\right\|_{H^k}$ and $\left\|\int_0^\tau(\tau-s) \mathcal{L} \mathcal{R}_2(s) \ds\right\|_{H^k}$, we only need to estimate $\|\frac{\partial^2 f(u(\xi))}{\partial \xi^2}\|_{H^k}$ and $\|\mathcal{L}\frac{\partial^2 f(u(\xi))}{\partial \xi^2}\|_{H^k}$ according to Lemma \ref{estimate_M}, where $\xi$ is a variable related to time. Therefore, from \eqref{eq:f_tt}, it is sufficient to estimate the following terms:
\begin{equation} \label{eq:terms}
    \begin{aligned}
     &\left\|\mathcal{L}^2u(\xi)-\mathcal{L}f(u(\xi))-f^{\prime}(u(\xi))\mathcal{L}u(\xi)+f^{\prime}(u(\xi))f(u(\xi))\right\|_{H^k} ,\\
        &\left\|u^2(\xi)\left(\mathcal{L}^2u(\xi)-\mathcal{L}f(u(\xi))-f^{\prime}(u(\xi))\mathcal{L}u(\xi)+f^{\prime}(u(\xi))f(u(\xi))\right)\right\|_{H^k} ,\\
        &\left\|u(\xi)\left(\mathcal{L}(u(\xi))-f(u(\xi))\right)^2\right\|_{H^k},\\
        &\left\|\mathcal{L}\left(\mathcal{L}^2u(\xi)-\mathcal{L}f(u(\xi))-f^{\prime}(u(\xi))\mathcal{L}u(\xi)+f^{\prime}(u(\xi))f(u(\xi))\right)\right\|_{H^k}, \\
        &\left\|\mathcal{L}\left(u^2(\xi)\left(\mathcal{L}^2u(\xi)-\mathcal{L}f(u(\xi))-f^{\prime}(u(\xi))\mathcal{L}u(\xi)+f^{\prime}(u(\xi))f(u(\xi))\right)\right)\right\|_{H^k} ,\\
        &\left\|\mathcal{L}\left(u(\xi)\left(\mathcal{L}(u(\xi))-f(u(\xi))\right)^2\right)\right\|_{H^k} .
    \end{aligned}
\end{equation}
Note that in the $\|\cdot\|_{H^k}$ norm, the highest order of differentiation with respect to spatial variable in these terms is $6$th order. The corresponding components are $\mathcal{L}^3 u(\xi)$ and $u^2(\xi) \mathcal{L}^3 u(\xi)$. The component $\mathcal{L}(u^2(\xi) \mathcal{L}^2 u(\xi))$ contains $5$th-order derivatives of $u$, while the remaining components involve derivatives of order up to $4$. 
According to the Sobolev embedding inequality $H^2\hookrightarrow L^\infty$, Sobolev multiplication inequality \eqref{multiplication} and Theorem \ref{lemma_derivative}, we can claim that all terms in \eqref{eq:terms} can be bounded by some constant depending on $(\Omega,d,T,\|v\|_{H^{k+6}},\varepsilon)$ (the detailed proof is too lengthy and we leave it to interested readers). Thus, there exists a constant $C_1>0$ depending only on $(\Omega,d,T,\|v\|_{H^{k+6}},\varepsilon)$, such that
\begin{equation} \label{eq:R2}
    \begin{aligned}
        \left\|\int_0^\tau \mathcal{R}_2(s) \ds+\int_0^\tau(\tau-s) \mathcal{L} \mathcal{R}_2(s) \ds\right\|_{H^k} \le C_1 \tau^3.
    \end{aligned}
\end{equation}
According to Theorem \ref{lemma_derivative} and \eqref{eq:R}, we have 
\begin{equation} \label{eq:estimate3}
\begin{aligned}
 \left\| \int_0^\tau \mathcal{R}_1(f(u(s))) \ds \right\|_{H^k}
\le & \int_0^\tau\left\|  \mathcal{R}_1(f(u(s))) \right\|_{H^k}\ds \\
\le & \int_0^\tau \frac{(\tau-s)^2}{2}\|\mathcal{L}^2\left(f\left(u\left(s\right)\right)\right)\|_{H^k} \ds \le C_2 \tau^3,
\end{aligned}
\end{equation}
where $C_2$ depends on $(\Omega,d,T,\|v\|_{H^{k+6}},\varepsilon)$.
Combining \eqref{eq:duhamel}, \eqref{eq:int_f}, \eqref{eq:int_lf}, \eqref{eq:R2} and \eqref{eq:estimate3}, we have the desired result
\begin{equation}
    \begin{aligned}
        & \left\| u(\tau) -   \SL v+\tau f(v)+\frac{\tau^2}{2} \mathcal{L} f(v)+\frac{\tau^2}{2} f^{\prime}(v) \mathcal{L} v-\frac{\tau^2}{2} f(v) f^{\prime}(v) \right\|_{H^k} \le C_3 \tau^3,
    \end{aligned}
\end{equation}
where $C_3$ only depends on $(\Omega,d,T,\|v\|_{H^{k+6}},\varepsilon)$.
\end{proof}

\begin{lemma} \label{Lemma_secondorder}
Assume $v \in H^{k+6}(\Omega)$, $\|v\|_{L^{\infty}}\le 1$, and $0<\tau \le 1$. There exists a constant $C$ only depending on $(\Omega,d,T,\|v\|_{H^{k+6}},\varepsilon)$, such that 
    \begin{equation}
        \left\| \mathcal{S}(\tau) v-\mathcal{T}(\tau)v\right\|_{H^k} \le C \tau^3,
    \end{equation}
    where $\mathcal{T}(\tau)$ is the exact solution operator to \eqref{T(tau)}.
\end{lemma}
\begin{proof}
    Lemmas \ref{convergence_1} and \ref{convergence_2} can directly lead to the conclusion.
\end{proof}

\begin{lemma} \label{polynomial_Hk}
Assume $\|v_1^0\|_{L^{\infty}}\le 1$ and $\|v_2^0\|_{L^{\infty}} \le 1$.
In the cases of $k=0$ or $k\ge 2$,    for any $v_1^0\in {H^{k}}(\Omega)$, $v_2^0\in {H^{k}}(\Omega)$, we have
    \begin{equation}
    \begin{aligned}
        \left\|\SN v_1^0-\SN v_2^0\right\|_{H^{k}}\le e^{C\tau}\left\| v_1^0- v_2^0\right\|_{H^{k}},
    \end{aligned}
\end{equation}
where $C$ depends on $(\Omega,d,T,\|v_1^0\|_{H^{k}},\|v_2^0\|_{H^{k}})$.  In the case of $k=1$, for any $v_1^0\in {H^{2}}(\Omega)$, $v_2^0\in {H^{2}}(\Omega)$, we have
    \begin{equation}
    \begin{aligned}
        \left\|\SN v_1^0-\SN v_2^0\right\|_{H^{1}}\le e^{\tilde{C}\tau}\left\| v_1^0- v_2^0\right\|_{H^{1}},
    \end{aligned}
\end{equation}
where $\tilde{C}$ depends on $(\Omega,d,T,\|v_1^0\|_{H^{2}},\|v_2^0\|_{H^{2}})$.
\end{lemma}
\begin{proof}
Suppose $v_i~(i=1,2)$ is the exact solution of the following equation:
\begin{equation} \label{eq:u-v}
\left\{\begin{array}{l}
\partial_t v_i=v_i-v_i^3,\quad t\in(0,\tau] \\
v_i(0)=v_i^0.
\end{array}\right.
\end{equation}
Acting $D^{\alpha}$ on both sides of \eqref{eq:u-v} for $i=1,2$, we have
\begin{equation} \label{eq:D(u-v)}
    \begin{aligned}
       \partial_t\left(D^\alpha\left(v_1-v_2\right)\right)=D^\alpha\left(v_1-v_2\right)-D^\alpha v_1^3+D^\alpha v_2^3.
    \end{aligned}
\end{equation}
Taking inner product with $D^\alpha\left(v_1-v_2\right)$ in \eqref{eq:D(u-v)}, we have
\begin{equation} \label{eq:v1-v2}
\begin{aligned}
&\frac{1}{2} \partial_t\left\|D^\alpha\left(v_1-v_2\right)\right\|_{L^2}^2  =\left\|D^\alpha\left(v_1-v_2\right)\right\|_{L^2}^2-\left\langle D^\alpha (v_1^3-v_2^3), D^\alpha\left(v_1-v_2\right)\right\rangle\\
 \leq&\left\|D^\alpha\left(v_1-v_2\right)\right\|_{L^2}^2+\left\|D^\alpha (v_1^3-v_2^3)\right\|_{L^2}\left\|D^\alpha\left(v_1-v_2\right)\right\|_{L^2} \\
 \leq&\left\|D^\alpha\left(v_1-v_2\right)\right\|_{L^2}^2+\left\|D^\alpha \left((v_1-v_2)(v_1^2+v_2^2+v_1v_2)\right)\right\|_{L^2}\left\|v_1-v_2\right\|_{H^k}.
\end{aligned}
\end{equation}

In the case of $k=0$, using the maximum principle, we have
\begin{equation}
    \begin{aligned} \label{eq:v1-v2(L2)}
       &\frac{1}{2} \partial_t\left\|v_1-v_2\right\|_{L^2}^2 \le 4\left\|v_1-v_2\right\|_{L^2}^2,
    \end{aligned}
\end{equation}
which leads to 
\begin{equation}
    \begin{aligned}
        \left\|\SN v_1^0-\SN v_2^0\right\|_{L^2}\le e^{4\tau}\left\| v_1^0- v_2^0\right\|_{L^2}.
    \end{aligned}
\end{equation}

In the case of $k=1$, we only need to consider $|\alpha|=1$. Using the maximum principle and  Sobolev embedding inequality for $d\le 3$:
\begin{equation}
W^{1, 2}(\Omega) \times W^{1, 2}(\Omega) \hookrightarrow W^{0, 2}(\Omega),
\end{equation}
we have
\begin{equation}
    \begin{aligned}
        &\left\|D^\alpha \left((v_1-v_2)(v_1^2+v_2^2+v_1v_2)\right)\right\|_{L^2} \\\le & 3\left\|D^\alpha \left(v_1-v_2\right)\right\|_{L^2}+\left\| (v_1-v_2)(2v_1D^\alpha v_1+2v_2D^\alpha v_2+v_2D^\alpha v_1+v_1D^\alpha v_2)\right\|_{L^2}\\
        \le & C_1\left\|v_1-v_2\right\|_{H^1},
    \end{aligned}
\end{equation}
where $C_1$ is a constant depending on $(\Omega,d,T,\|v_1^0\|_{H^2},\|v_2^0\|_{H^2})$ from the regularity analysis of \eqref{eq:u-v} as aforementioned. Therefore, from \eqref{eq:v1-v2}, we have
\begin{equation} 
\begin{aligned}
&\frac{1}{2} \partial_t\left\|D^\alpha\left(v_1-v_2\right)\right\|_{L^2}^2  
\le \left\|D^\alpha\left(v_1-v_2\right)\right\|_{L^2}^2+C_1\left\|v_1-v_2\right\|_{H^1}^2.
\end{aligned}
\end{equation}
Combining it with \eqref{eq:v1-v2(L2)}, we have
\begin{equation}
    \begin{aligned} \label{eq:v1-v2(H1)}
      &\frac{1}{2} \partial_t\left\|v_1-v_2\right\|_{H^1}^2  
\le C_2\left\|v_1-v_2\right\|_{H^1}^2, 
    \end{aligned}
\end{equation}
where $C_2$ is a constant depending on $(\Omega,d,T,\|v_1^0\|_{H^2},\|v_2^0\|_{H^2})$. Then, we can conclude that
\begin{equation}
    \begin{aligned}
        \left\|\SN v_1^0-\SN v_2^0\right\|_{H^{1}}\le e^{C_2\tau}\left\| v_1^0- v_2^0\right\|_{H^{1}}.
    \end{aligned}
\end{equation}

In the case of $k\ge 2$, using the multiplication theorem \eqref{multiplication} and the regularity analysis of the solution to \eqref{eq:u-v}, we have
\begin{equation}
    \begin{aligned}
        &\left\|D^\alpha \left((v_1-v_2)(v_1^2+v_2^2+v_1v_2)\right)\right\|_{L^2} \\
        \le &\left\|(v_1-v_2)(v_1^2+v_2^2+v_1v_2)\right\|_{H^k} \\
        \le
        & C_3\left\|v_1-v_2\right\|_{H^{k}}\left(\left\|v_1\right\|_{H^{k}}^2+\left\|v_2\right\|_{H^{k}}^2+\left\|v_1\right\|_{H^{k}}\left\|v_2\right\|_{H^{k}}\right)\\
        \le & C_4\left\|v_1-v_2\right\|_{H^{k}},
    \end{aligned}
\end{equation}
where $C_3$ depends on $(\Omega,d)$ and $C_4$ depends on $(\Omega,d,T,\|v_1^0\|_{H^{k}},\|v_2^0\|_{H^{k}})$. From \eqref{eq:v1-v2}, we have
\begin{equation} 
\begin{aligned} \label{eq:v1-v2(Hk)}
&\frac{1}{2} \partial_t\left\|D^\alpha\left(v_1-v_2\right)\right\|_{L^2}^2  
\le \left\|D^\alpha\left(v_1-v_2\right)\right\|_{L^2}^2+C_4\left\|v_1-v_2\right\|_{H^{k}}^2.
\end{aligned}
\end{equation}
Summing up \eqref{eq:v1-v2(Hk)} for $|\alpha|\le k$, we have
\begin{equation}
    \begin{aligned}
      &\frac{1}{2} \partial_t\left\|v_1-v_2\right\|_{H^{k}}^2  
\le C_5\left\|v_1-v_2\right\|_{H^{k}}^2, 
    \end{aligned}
\end{equation}
where $C_5$ is a constant depending on $(\Omega,d,T,\|v_1^0\|_{H^{k}},\|v_2^0\|_{H^{k}})$. Therefore, we have
\begin{equation}
    \begin{aligned}
        \left\|\SN v_1^0-\SN v_2^0\right\|_{H^{k}}\le e^{C_5\tau}\left\| v_1^0- v_2^0\right\|_{H^{k}}.
    \end{aligned}
\end{equation}
\end{proof}

Now we are ready to give the $H^k$-norm error estimate for the Strang splitting method as follows.
\begin{theorem} \label{theorem_tau}
    Assume $u^0 \in H^{k+6}(\Omega)$, $\|u^0\|_{L^{\infty}}\le 1$, and $\tau_{\max}\le 1$, where $\tau_{\max}$ is defined in \eqref{eq:tau_max}. There exists a constant $C$, depending on $(\Omega,d,T,\|u^0\|_{H^{k+6}},\varepsilon)$, such that
    \begin{align}
        \sup _{\tau_1 +\tau_2+\cdots \tau_n \leq T}\left\|u^n-u(t_n)\right\|_{H^k} \leq C \tau_{\max}^2.
    \end{align}
\end{theorem}
\begin{proof}
By the triangle inequality, the $H^k$-norm error between the numerical solution and the exact solution is 
\begin{align} \label{eq:triangle}
    \left\|u^{n+1} - u(t_{n+1})\right\|_{H^k} \le \left\|\mathcal{S}(\tau_{n+1})u^n-\mathcal{S}(\tau_{n+1})u(t_n)\right\|_{H^k}+\left\| \mathcal{S}(\tau_{n+1})u(t_n)-u(t_{n+1})\right\|_{H^k}.
\end{align}
Using Lemma \ref{lemma_derivative} and  \ref{Lemma_secondorder}, we obtain
\begin{equation}
    \begin{aligned}
       \left\| \mathcal{S}(\tau_{n+1})u(t_n)-u(t_{n+1})\right\|_{H^k} \le C_1 \tau_{n+1}^3 , \label{eq:convergence1}
    \end{aligned}
\end{equation}
where $C_1$ depends on $(\Omega,d,T,\|u^0\|_{H^{k+6}},\varepsilon)$.
Next, from Lemma \ref{polynomial_Hk}, we have the following estimate
\begin{equation}
    \begin{aligned}
        &\left\| \mathcal{S}(\tau_{n+1})u^n-\mathcal{S}(\tau_{n+1})u(t_n)\right\|_{H^k} \le \left\| \mathcal{S}_{\mathcal{N}}(\tau_{n+1})\SLt u^n-\mathcal{S}_{\mathcal{N}}(\tau_{n+1})\SLt u(t_n)\right\|_{H^k}\\
        \le& e^{ C_2\tau_{n+1}}\left\|\SLt u^n - \SLt u(t_n)\right\|_{H^k}\le e^{ C_2\tau_{n+1}}\left\|u^n - u(t_n)\right\|_{H^k},\label{eq:convergence2}
    \end{aligned}
\end{equation}
where $C_2$ depends on $(\Omega,d,T,\|u^0\|_{H^{k+6}},\varepsilon)$.
Combining \eqref{eq:convergence1} and \eqref{eq:convergence2}, we derive
    \begin{equation}
        \begin{aligned}
            \left\|u^{n} - u(t_{n})\right\|_{H^k} &\le e^{C_2\tau_{n}}\left\|u^{n-1} - u(t_{n-1})\right\|_{H^k} + C_1\tau_{n}^3\\
            &\le e^{C_2\tau_{n}}(e^{C_2\tau_{n-1}}\left\|u^{n-2} - u(t_{n-2})\right\|_{H^k}+C_1\tau_{n-1}^3) + C_1\tau_{n}^3\\
            &= e^{C_2(\tau_{n-1}+\tau_{n})}\left\|u^{n-2} - u(t_{n-2})\right\|_{H^k} + C_1(\tau_{n}^3+e^{C_2\tau_{n}}\tau_{n-1}^3)\\
            & \cdots\\
            &\le e^{C_2(\tau_{n}+\cdots+\tau_1)}\left\|u^{0} - u(t_0)\right\|_{H^k} + C_1(\tau_{n}^3+e^{C_2\tau_{n}}\tau_{n-1}^3+\cdots+e^{C_2(\tau_{n}+\cdots+\tau_2)}\tau_{1}^3)\\
            &\le C_1T e^{C_2T} \tau_{\max}^2,
        \end{aligned} 
    \end{equation}
    which provides the desired error estimate 
    \begin{align}
        \sup _{\tau_1 +\tau_2+\cdots \tau_n \leq T}\left\|u^n-u(t_n)\right\|_{H^k} \le C \tau_{\max}^2,
    \end{align}
    where $C$ depends on $(\Omega,d,T,\|u^0\|_{H^{k+6}},\varepsilon)$.
\end{proof}
\color{black}
\begin{remark}
    In \cite{li_stability_2022-1}, $u^0 \in H^{40}(\Omega)$ is required to prove the second-order approximation.
    However, we select a different triangle inequality \eqref{eq:triangle} in this article such that $u^0 \in H^{6}(\Omega)$ is sufficient for the second-order approximation. This is a relaxation for the initial conditions. 
\end{remark}

\begin{remark}
Consider the Strang splitting method with variable time steps for the Allen--Cahn equation with logarithmic potential and homogeneous Neumann boundary condition. In this case, there are two challenges compared to the polynomial case. Firstly, the nonlinear solution operator $\mathcal{S}_{\mathcal{N}}$ can not be given explicitly. One feasible way is to approximate $\mathcal{S}_{\mathcal{N}}$ by Runge--Kutta method, which, however, might cause problems in proving the maximum principle and energy stability. Secondly, the Lipschitz constant of the nonlinearity is large near the maximum bound, which might lead to a large error estimate coefficient. We will focus on these issues in future works.
\end{remark}

\begin{remark}
    The $H^k$-norm error analysis for the case of homogeneous Neumann boundary conditions can be directly extended to the case of periodic boundary conditions. The above theorems also hold for periodic boundary conditions, the case of which is even easier to prove.
\end{remark}

\section{Numerical experiments}
    \label{simulations}
    In this section, we show the numerical results of the Strang splitting method on the domain $\Omega = [0,2\pi]^2$. When solving the linear solution operator $S_{\mathcal{L}}(\tau)$, we use the pseudo-spectral method for space discretization \cite{ju2015fast}. 
    Note that the efficiency of operator splitting could be enhanced by employing variable time steps, without compromising the accuracy of the results. For the Hamiltonian systems, performing some time variable transformations and then applying a constant step size implementation to the transformed system could be more efficient \cite{blanes2005adaptive,Hairer2006,blanes2012explicit}. For the dissipative systems, in  \cite{qiao_adaptive_2011}, Qiao, Zhang and Tang propose two adaptive strategies based on the energy variation and the solution roughness respectively, which can be applied to the operator splitting method for molecular beam epitaxy model \cite{cheng2015fast}. In \cite{gomez2011provably}, Gomez and Hughes develop another adaptive strategy, where they adjust the time step based on whether the computed error is within the tolerance range. 
 These adaptive strategies can be applied to solve the Allen--Cahn equation \cite{hou2023linear,shen2016maximum}, in particular the operator splitting method \cite{huang2019adaptive}. The efficiency of adaptive strategy is presented in the following. 

\begin{exmp}\label{accuracy_test}
    Consider the Allen--Cahn equation with the polynomial potential $f(u)=u^3-u$, where $\varepsilon = 0.1$. We take
the initial condition consisting of seven circles:
  \begin{equation}
    u_0(x, y)=-1+\sum_{i=1}^7 f_0\left(\sqrt{\left(x-x_i\right)^2+\left(y-y_i\right)^2}-r_i\right),
  \end{equation}
  where the centers and radii are given by
  \begin{equation}
    \begin{array}{l|lllllll}
    \hline i & 1 & 2 & 3 & 4 & 5 & 6 & 7 \\
    \hline x_i & \pi / 2 & \pi / 4 & \pi / 2 & \pi & 3 \pi / 2 & \pi & 3 \pi / 2 \\
    y_i & \pi / 2 & 3 \pi / 4 & 5 \pi / 4 & \pi / 4 & \pi / 4 & \pi & 3 \pi / 2 \\
    r_i & \pi / 5 & 2 \pi / 15 & 2 \pi / 15 & \pi / 10 & \pi / 10 & \pi / 4 & \pi / 4 \\
    \hline
    \end{array}
    \end{equation}
  and $f_0$ is defined by
  \begin{equation}
    f_0(s)= \begin{cases}2 e^{-\varepsilon^2 / s^2}, & \text { if } s<0, \\ 0, & \text { otherwise. }\end{cases}
  \end{equation}
\end{exmp}

We use $512\times 512$ Fourier modes for the space discretization. The homogeneous Neumann boundary condition is employed. A tiny time step size $\tau=0.0001$ is used to calculate the ``reference'' solution $U_{\text{ref}}$ at $T_1=1$. The time step is chosen to be random. We take
$$\tau_k := \frac{\sigma_k T_1}{\sum_{k=1}^N\sigma_k}, \quad 1\le k \le N,$$ where $\sigma_k$ is a random number uniformly distributed in $[0,1]$ and $N$ is the number of subintervals. The $H^1$-norm error is defined as
\begin{equation}\label{eq:order}
    \begin{aligned}
        e(N):=\left\|U_{\text{ref}}(T_1)-u^N\right\|_{H^1},
    \end{aligned}
\end{equation}
where $u^N$ is the numerical solution at $t=T_1$ using $N$ subintervals. The rate of convergence is computed as
$$\text{rate} \approx \frac{\log(e(N)/e(2N))}{\log(\tau_{N,\max}/\tau_{2N,\max})},$$ where $\tau_{N,\max}$ is the maximum time step in total $N$ steps and so is $\tau_{2N,\max}$. Table \ref{convergence_poly} shows $H^1$-norm error and convergence rate, where the convergence rate is about $\mathcal{O}(\tau_{\max}^2)$.
\begin{table}[h!]
\renewcommand\arraystretch{1.2}
\begin{center}
\def\temptablewidth{1\textwidth}
\caption{Example \ref{accuracy_test}: $H^1$-norm errors of numerical solution at time $T_1 = 1$ for the Allen--Cahn equation with polynomial potential.}
{\rule{\temptablewidth}{0.5pt}}
\begin{tabular*}{\temptablewidth}{@{\extracolsep{\fill}}cccccc}
   $N$   &$200$     &$400$     &$800$
   &$1600$  & $3200$ \\  \hline
  $H^1$-error   & $4.913\times 10^{-5}$   & $1.228\times 10^{-5}$   & $3.042\times 10^{-6}$ & $7.168\times 10^{-7}$ & $1.781\times 10^{-7}$\\[3pt]
rate  & -- &$1.9595$ &$1.9697$ &$2.0515$ & $2.0412$
\end{tabular*}
{\rule{\temptablewidth}{0.5pt}}
\end{center}
\label{convergence_poly}
\end{table}

We then conduct the long time simulations with variable time steps. We apply the following time stepping strategy as \cite{qiao_adaptive_2011}:
  \begin{equation} \label{eq:adaptive}
      \begin{aligned}
          \tau =\max \left(\tau_{\min }, \frac{\tau_{\max }}{\sqrt{1+\alpha\left|E'(t)\right|^2}}\right),
      \end{aligned}
  \end{equation}
where $\tau_{\min}$ and $\tau_{\max}$ are given minimum and maximum steps, and $\alpha>0$ is some constant.
Consequently, fast decay of energy will lead to small time steps, while slow decay of energy (meaning slow change of interface) leads to large time steps.

We compute the ``reference'' solution $u_{\text{ref}}$ at $T_2=10$ with uniform time step $\tau =0.001$. In this test, we choose $\tau_{\min}=0.001$, $\tau_{\max}=0.1$ and $\alpha=100$ to calculate the numerical solution $u_{\text{num}}$. We compute the relative error of $u_{\text{num}}$ as
\begin{equation} \label{eq:err}
    \begin{aligned}
e_{\text{rel}}(T_2)\coloneqq\frac{\left\|u_{\text{ref}}(T_2)-u_{\text{num}}(T_2)\right\|_{H^1}}{\left\|u_{\text{ref}}(T_2)\right\|_{H^1}},
    \end{aligned}
\end{equation}
where $u_{\text{ref}}(T_2)$ and $u_{\text{num}}(T_2)$ represent the solution of $u_{\text{ref}}$ and $u_{\text{num}}$ at $T_2=10$.

 In Fig. \ref{pic2}, it is observed that the $L^{\infty}$-norm of the numerical solutions are bounded by $1$ and the energy is dissipating. In Fig. \ref{pic2}(d), we can observe that the CPU cost of adaptive strategy is much less than that of the uniform-time-step strategy. Meanwhile, the relative error is $e_{\text{rel}}(T_2)\approx 2.5\times10^{-12}$. It indicates that the adaptive strategy is almost as accurate as the uniform-time-step strategy, but more efficient. 
 
\begin{figure}[htbp]
 \centering
\includegraphics[trim = {0in 0.2in 0.in 0},clip,width = 1.0\textwidth]{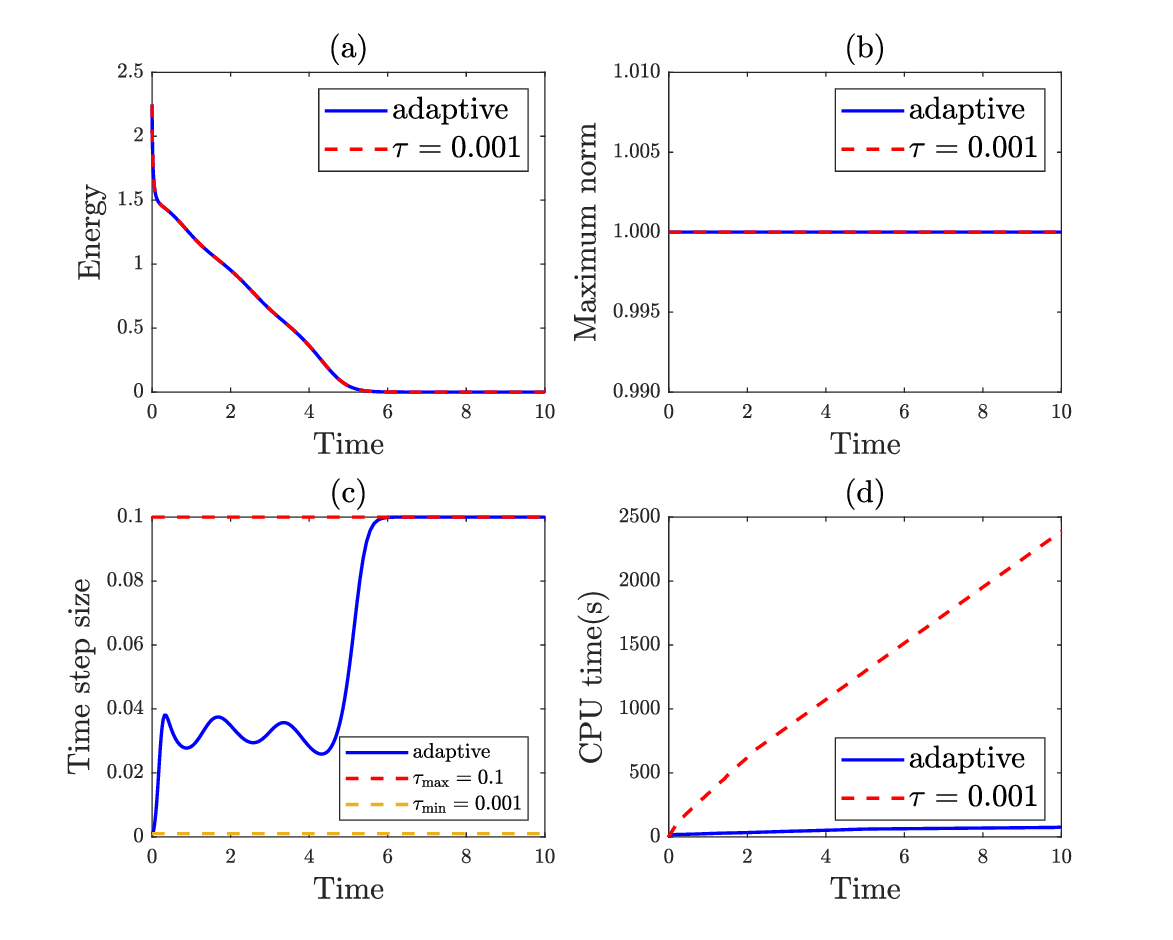}
    \caption{Example \ref{accuracy_test}: (a) Evolution of energy with final time $T_2=10$. (b) Evolution of maximum norm with final time $T_2=10$. (c) Evolution of time step size with final time $T_2=10$. (d) Evolution of CPU cost time with final time $T_2=10$.}
      \label{pic2}
    \end{figure}

\begin{exmp}
\label{example 3}
We consider the Allen--Cahn equation with logarithmic potential, i.e.,
\begin{equation}
    \begin{aligned}
        \partial_t u=\varepsilon^2 \Delta u+\theta_c u-\frac{\theta}{2}(\ln (1+u)-\ln (1-u)).
    \end{aligned}
\end{equation}
We take the initial condition as
$
        u^0(x,y) = 0.5\left( \chi \left(\left(x-\pi\right)^2+\left(y-\pi\right)^2 \le 1.2\right)-0.5\right).
$
The other parameters are $\varepsilon=0.01$, $\theta_c=1$ and $\theta=\frac{1}{4}$ as in \cite{li_stability_2022-1}. 
\end{exmp}

The Strang splitting method can be written as 
\begin{equation}
    \begin{aligned}
        u^{n+1} = \mathcal{S}_{\mathcal{L}}\left(\frac{\tau_{n+1}}{2}\right)\mathcal{S}_{\mathcal{N}}^{(\text{LOG})}(\tau_{n+1})\mathcal{S}_{\mathcal{L}}\left(\frac{\tau_{n+1}}{2}\right)u^n,
    \end{aligned}
\end{equation}
where Runge--Kutta formulae is applied in approximating $\mathcal{S}_{\mathcal{N}}^{(\text{LOG})}(\tau_{n+1})$.
The Butcher tableau is written as
\begin{equation}
\begin{array}{c|cc}
a & a & 0 \\
1-a & 1-2 a & a \\
\hline & \frac{1}{2} & \frac{1}{2}
\end{array}
\end{equation}
with $a=1+\frac{\sqrt{2}}{2}$.
We use $512\times512$ Fourier modes for the space discretization. The homogeneous Neumann boundary condition is employed.
First, we show the $H^1$-norm errors of the numerical solution at $T_1=1$ using \eqref{eq:order}. Table \ref{convergence_log}
shows that the convergence order is about 2. Next, a numerical experiment with $512\times512$ Fourier modes up to $T_2=10$ is carried out. We choose $\tau_{\min}=0.001$, $\tau_{\max}=0.01$ and $\alpha=100$ to calculate the solution $u_{\text{num}}$ with the same adaptive strategy \eqref{eq:adaptive}. We also calculate the ``reference'' solution $u_{\text{ref}}$ with time step $\tau=0.001$.  
In Fig. \ref{pic3}, the $L^{\infty}$-norm of the numerical solutions remains bounded, and the energy dissipates over time. In Fig. \ref{pic3}(d), the CPU cost of adaptive strategy is much less than that of the uniform-time-step strategy. By \eqref{eq:err}, the relative error is about $5.1\times10^{-5}$. It indicates that the adaptive strategy is almost as accurate as the uniform-time-step strategy, but more efficient.
\begin{table}[htbp]
    \renewcommand\arraystretch{1.2}
    \begin{center}
    \def\temptablewidth{1\textwidth}
    \caption{Example \ref{example 3}: $H^1$-norm errors of numerical solutions at time $T_2 = 1$ for Allen--Cahn equation with logarithmic potential.}\label{tab3}
    {\rule{\temptablewidth}{0.5pt}}
    \begin{tabular*}{\temptablewidth}{@{\extracolsep{\fill}}cccccc}
       $N$   &$100$     &$200$     &$400$
       &$800$  & $1600$ \\  \hline
      $H^1$-error   & $1.300\times 10^{-3}$   & $3.228\times 10^{-4}$   & $7.240\times 10^{-5}$ & $1.796\times 10^{-5}$ & $4.468\times 10^{-6}$\\[3pt]
    rate  & -- &$2.0002$ &$2.0189$ &$2.0331$ & $2.0211$
    \end{tabular*}
    {\rule{\temptablewidth}{0.5pt}}
    \end{center}
    \label{convergence_log}
    \end{table}

    \begin{figure}[htbp]
        \centering
       \includegraphics[trim = {0in 0.2in 0.in 0.0in},clip,width = 1.0\textwidth]{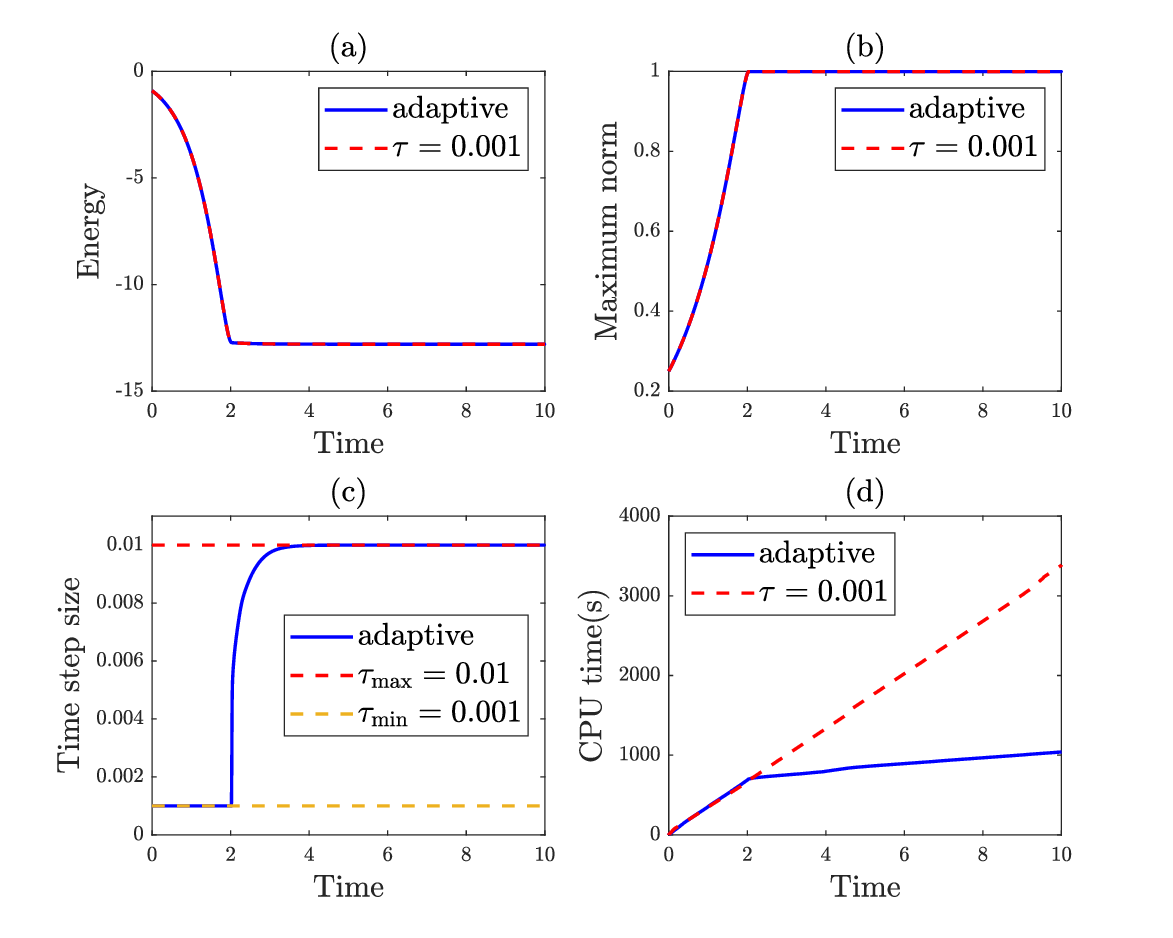}
        \caption{Example \ref{example 3}: (a) Evolution of energy with final time $T_2=10$. (b) Evolution of maximum norm with final time $T_2=10$. (c) Evolution of CPU cost with final time $T_2=10$. (d) Evolution of time step size with final time $T_2=10$. }
        \label{pic3}
    \end{figure}

\begin{exmp}
\label{example 4}
Consider the ternary conservative Allen--Cahn equations:
\begin{equation}
\partial_t u_l=\varepsilon^2 \Delta u_l-\left(f\left(u_l\right)-\beta\left(u_l\right)+\Lambda\left(u_1, u_2, u_3\right)\right) \quad \text { in } \Omega \times(0, T], \quad l=1,2,3 \text {, }
\end{equation}
where 
$
F\left(u_l\right)=\frac{1}{2} u_l^2\left(1-u_l\right)^2, f\left(u_l\right)=F^{\prime}\left(u_l\right)
$ and $\varepsilon = 0.05$. $\beta(u)$ and $\Lambda\left(u_1, u_2, u_3\right)$ are defined as 
\begin{align}
\beta(u_l)&=\frac{1}{\left|\Omega\right|} \int_{\Omega} F^{\prime}(u_l) \dx, \quad l=1,2,3,\quad 
\Lambda\left(u_1, u_2, u_3\right)=-\frac{1}{3} \sum_{l=1}^3\left(f\left(u_l\right)-\beta\left(u_l\right)\right).\label{eq:hyperplanelink}
\end{align}
The energy functional can be written as 
$
    E(u_1,u_2,u_3)=\sum_{l=1}^3\int_{\Omega} \left(\frac{\varepsilon^2}{2}\left|\nabla u_l\right|^2+F(u_l)\right)\dx.
$
We take the following initial conditions
\begin{equation}
\left\{\begin{array}{l}
\phi_l(x)=\operatorname{rand}(x), \\
u_l^0(x)=\frac{\phi_l(x)}{\phi_1(x)+\phi_2(x)+\phi_3(x)},
\end{array}\right.
\end{equation}
where $l=1,2,3$ and $\operatorname{rand}(\cdot)$ is the uniformly distributed random function.   
\end{exmp}

Here, the nonlinear operator 
$\SN (u_1,u_2,u_3)$ is approximated using the Runge--Kutta formula (as in \cite{weng_stability_2024}).
The first equation in \eqref{eq:hyperplanelink} guarantees the mass conservation of each component and the second equation ensures the hyperplane link $u_1+u_2+u_3=1$. 

We use $128\times 128$ Fourier modes for the space discretization. The periodic boundary condition is employed. We choose $\tau_{\min}=0.001$, $\tau_{\max}=0.1$ and $\alpha=100$ to calculate the numerical solution. In Fig. \ref{pic4}, the equilibrium solution exhibits a regular shape with the contact angles of about $\frac{2}{3}\pi$. The evolution of time step size and original energy is displayed in Fig. \ref{pic5}. We can see that the original energy is dissipating over time. 
\begin{figure}[htbp]
     \centering
       \includegraphics[trim = {0.0in 0.in 0.in 0.0in},clip,width = 1.00\textwidth]{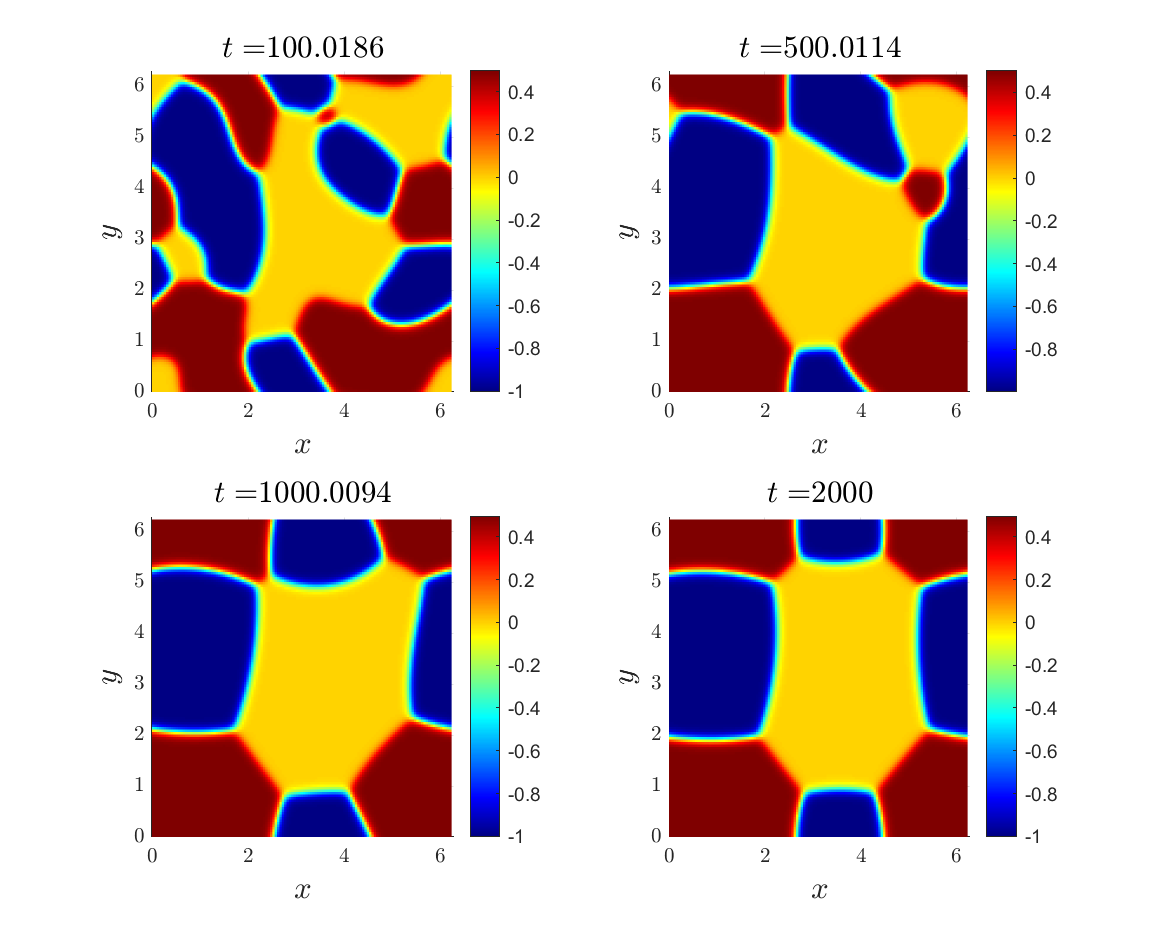}
      \caption{Example \ref{example 4}: Snapshots of $\frac{1}{2}u_1-u_2$ at $t=100.0186,500.0114,1000.0094,2000$ using adaptive time-stepping strategy \eqref{eq:adaptive}.}
      \label{pic4}
\end{figure}

\begin{figure}[htbp]
    \centering
     \includegraphics[trim = {0.in 0.in 0.in 0},clip,width = 1.0\textwidth]{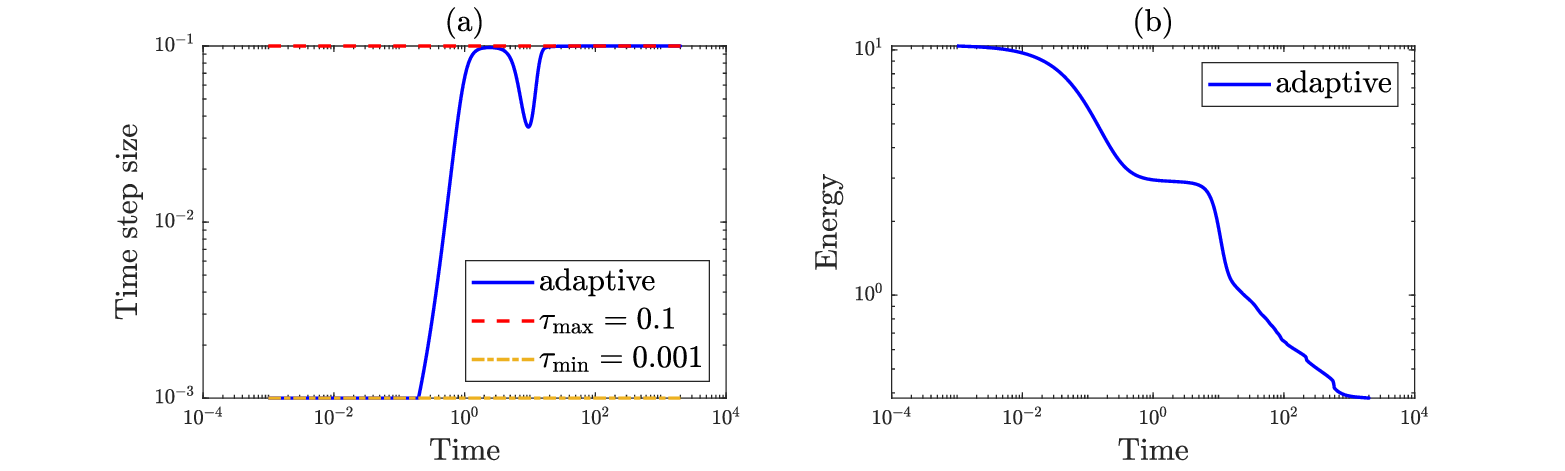}
      \caption{Example \ref{example 4}: (a) Evolution of time step size (in log scale) with final time $T=2000$. (b) Evolution of energy (in log scale) with final time $T=2000$.}
     \label{pic5}
\end{figure}

 \section{Conclusions} \label{section_conclusion}
    In this article, we consider to solve the Allen--Cahn equation with homogeneous Neumann boundary condtion by the Strang splitting method with variable time steps. We establish the $H^k$-norm stability of Strang splitting method for the Allen--Cahn equation with polynomial potential. Furthermore, rigorous $H^k$-norm convergence analysis are given, under the initial regularity assumptions $u^0\in H^{k+6}(\Omega)$. Numerical simulations show the convergence rate, energy dissipation law and the efficiency of the adaptive time-stepping strategy.
    


\section*{Acknowledgments}
 C. Quan is supported by National Natural Science Foundation of China  (Grant No. 12271241), Guangdong Provincial Key Laboratory of Mathematical Foundations for Artificial Intelligence (2023B1212010001), Guangdong Basic and Applied Basic Research Foundation (Grant No. 2023B1515020030), and Shenzhen Science and Technology Innovation Program (Grant No. JCYJ20230807092402004). Z. Tan is supported by the National Nature Science Foundation of China (12371418), Guangdong Natural Science Foundation (2024A1515010694,2022A1515010426), and Guangdong Province Key Laboratory of Computational Science at the Sun Yat-sen University (2020B1212060032). 

\section*{Data Availibility}
Data will be made available on reasonable request.

\bibliographystyle{siamplain}
\bibliography{references}
\end{document}